\documentclass{scrartcl}
\usepackage[numbers]{natbib}
\usepackage{fontenc}
\usepackage{shadethm}
\usepackage{amsthm}
\usepackage{float}
\usepackage{bm}
\usepackage{mathtools}
\usepackage{graphicx}
\usepackage{subfig}
\usepackage{stmaryrd}
\usepackage{mathrsfs}
\usepackage{amsmath}
\usepackage{amssymb}
\usepackage{wasysym}
\usepackage{sectsty}
\usepackage{stackengine}
\usepackage[hyperfootnotes=false]{hyperref}
\usepackage[a4paper, left=2.7cm, right=2.7cm, top=2.5cm]{geometry}
\usepackage{chngcntr}
\counterwithin*{equation}{section}

\usepackage{cleveref}
\DeclareMathAlphabet{\mathpzc}{OT1}{pzc}{m}{it}

\sectionfont{\normalsize\centering}
\subsectionfont{\footnotesize\centering}

\theoremstyle{plain}
\newtheorem{thm}{Theorem}[section] 

\theoremstyle{definition}
\newtheorem{lem}[thm]{Lemma}
\newtheorem{quest}[thm]{Question}
\newtheorem{prop}[thm]{Proposition}

\newtheorem{cor}[thm]{Corollary}

\def\XXint#1#2#3{{\setbox0=\hbox{$#1{#2#3}{\int}$ }
		\vcenter{\hbox{$#2#3$ }}\kern-.6\wd0}}

\usepackage[utf8]{inputenc}
\usepackage[german,english,russian]{babel}

\newcounter{MPequ}

\newcounter{AppA}
\newenvironment{AppA}
{\stepcounter{AppA}%
	\addtocounter{equation}{0}%
	\equation}
{\endequation}
\newcounter{AppB}
\newenvironment{AppB}
{\stepcounter{AppB}%
	\addtocounter{equation}{0}%
	\equation}
{\endequation}
\newcounter{AppC}

\newcounter{AppD}

\newcounter{AppE}

\begin{document}\selectlanguage{english}
\begin{center}
\normalsize \textbf{\textsf{Non-optimal domains for the helicity maximisation problem}}
\end{center}
\begin{center}
 Wadim Gerner\footnote{\textit{E-mail address:} \href{mailto:wadim.gerner@edu.unige.it}{wadim.gerner@edu.unige.it}}
\end{center}
\begin{center}
{\footnotesize MaLGa Center, Department of Mathematics, Department of Excellence 2023-2027, University of Genoa, Via Dodecaneso 35, 16146 Genova, Italy}
\end{center}
{\small \textbf{Abstract:} 
In [J. Cantarella, D. DeTurck, H. Gluck and M. Teytel, J. Math. Phys. 41:5615 (2000)] the helicity isoperimetric problem which asks to find a smooth domain of fixed volume which maximises Biot-Savart helicity among all other smooth domains of fixed volume was initiated. It was shown that if an optimal domain exists, all of its boundary components must be tori.

The present work extends these results by establishing additional geometric constraints which optimal domains, if they exist, must satisfy. This allows to rule out the optimality of a broad class of solid tori. The existence of optimal domains remains an open problem.
\section{Introduction}
The main equations governing ideal magnetohydrodynamics in a bounded smooth domain $\Omega\subset\mathbb{R}^3$ are the following, cf. \cite[Chapter III Remark 1.1]{AK21},
\begin{gather}
	\label{S1E1}
	\partial_tv+\nabla_vv-\nu \Delta v+\nabla p=\operatorname{curl}(B)\times B\text{, }\partial_tB=[B,v]\text{, }\operatorname{div}(B)=0=\operatorname{div}(v)\text{, }B\parallel \partial \Omega\text{, }v|_{\partial\Omega}=0
\end{gather}
where $v$ denotes the fluid velocity, $B$ denotes the magnetic field, $\nu$ denotes the kinematic viscosity, $[B,v]$ denotes the Lie-bracket of $B$ and $v$, $\nabla_vv$ denotes the co-variant derivative of $v$ w.r.t. $v$, or equivalently $\nabla_vv=(Dv)\cdot v$ where $Dv$ denotes the Jacobian of $v$ and we impose tangent boundary conditions on $B$ and the no-slip boundary condition on the fluid velocity $v$. One important observation is that the equation $\partial_tB=[B,v]$ implies, \cite[Porposition 22.15]{L12},
\begin{gather}
	\label{S1E2}
	B=(\psi_t)_*B_0
\end{gather}
where $\psi_t$ denotes the flow of the (time-dependent) velocity field $v$, $(\psi_t)_*$ denotes the pushforward by $\psi_t$ and $B_0(x):=B(0,x)$ is the initial magnetic field. Since $v$ is div-free, (\ref{S1E1}), its flow $\psi_t$ is volume-preserving so that (\ref{S1E2}) tells us that the magnetic field $B$ at time $t$ can be obtained from $B_0$ by means of the action of an appropriate volume-preserving diffeomorphism on the initial field $B_0$.

Before we come to an implication of this observation we introduce the concept of helicity. The helicity of a vector field $B$ which is div-free and tangent to the boundary of a domain $\Omega$, can be defined as follows
\begin{gather}
	\label{S1E3}
	\mathcal{H}(B):=\frac{1}{4\pi}\int_{\Omega\times\Omega}B(x)\cdot\left(B(y)\times \frac{x-y}{|x-y|^3}\right)d^3yd^3x.
\end{gather}
Of importance in this context is the Biot-Savart operator, which can be defined as follows
\begin{gather}
	\label{S1E4}
	\operatorname{BS}(B)(x):=\frac{1}{4\pi}\int_{\Omega}B(y)\times \frac{x-y}{|x-y|^3}d^3y.
\end{gather}
With this definition, (\ref{S1E3}) may be equivalently expressed as $\mathcal{H}(B)=\langle B,\operatorname{BS}(B)\rangle_{L^2(\Omega)}\equiv \int_{\Omega}B(x)\cdot \operatorname{BS}(B)(x)d^3x$. The key observation is that helicity is preserved by the action of volume-preserving diffeomorphisms which are connected to the identity, cf. \cite[Theorem A]{CDGT002}, \cite[Proof of Corollary 2.7]{G24OptimalHelicityDomain}, i.e. if $(\psi_t)_{0\leq t\leq 1}$ is a smooth family of volume-preserving diffeomorphisms with $\psi_0=\operatorname{Id}$ and $B_0$ is a time-independent, div-free field, tangent to $\partial\Omega$, then
\begin{gather}
	\label{S1E5}
	\mathcal{H}((\psi_t)_*B_0)=\mathcal{H}(B_0)\text{ for all }t.
\end{gather}
This, in combination with (\ref{S1E2}) implies that helicity is a conserved quantity in the realm of ideal magnetohydrodynamics. Physically helicity may be regarded as a measure of the average linking of distinct field lines of the underlying magnetic field, \cite{Mo69},\cite{Arnold2014},\cite{V03}.

Further, helicity provides an obstruction to energy dissipation in the sense that helicity serves as a lower bound for the magnetic energy, cf. \cite[Theorem B]{CDG00},
\begin{gather}
	\label{S1E6}
	|\mathcal{H}(B)|\leq \sqrt[3]{\frac{3}{4\pi}}(\operatorname{vol}(\Omega))^{\frac{1}{3}}\|B\|^2_{L^2(\Omega)}.
\end{gather}
Exploiting Sobolev inequalities one can obtain the tighter bound, \cite[Equation (1.9)]{FrHe91I},
\begin{gather}
	\label{S1E7}
	|\mathcal{H}(B)|\leq \sqrt[3]{\frac{\pi}{16}}(\operatorname{vol}(\Omega))^{\frac{1}{3}}\|B\|^2_{L^2(\Omega)}.
\end{gather}
It was further informally argued by Woltjer, \cite{W58}, that a limit configuration $B_{\infty}=\lim_{t\rightarrow\infty}B(t,\cdot)$ should correspond to a magnetic field of minimal magnetic energy under a fixed helicity constraint.

One problem of interest in ideal magnetohydrodynamics is therefore to identify all those stationary magnetic fields $B$ which, in a fixed helicity class $\{\mathcal{H}=h\}$, $h\in \mathbb{R}$, minimise the magnetic energy, i.e. the square of the $L^2$-norm. This is a classical topic and was investigated in different settings, \cite{W58},\cite{Arnold2014},\cite{AL91},\cite{G20}. It is standard that minimising the magnetic energy in a fixed (positive) helicity class and maximising helicity in a fixed energy class are equivalent problems. The helicity maximisation problem has been investigated in \cite[Section 9]{CDG00} where it was shown that a div-free field $B$ tangent to $\partial\Omega$ maximises $\mathcal{H}(B)$ in its own energy-class if and only if $B$ is the eigenfield to the largest positive eigenvalue of a modified Biot-Savart operator.

More precisely, if we let
\begin{gather}
	\label{S1E8}
	\pi:L^2(\Omega,\mathbb{R}^3)\rightarrow L^2\mathcal{V}_0(\Omega):=\{B\in L^2(\Omega,\mathbb{R}^3)\mid \operatorname{div}(B)=0=\mathcal{N}\cdot B\}
\end{gather}
denote the $L^2$-orthogonal projection from the space of square-integrable vector fields onto the space of div-free fields which are tangent to the boundary (here $\mathcal{N}$ denotes the outward unit normal), then
\begin{gather}
	\label{S1E9}
	\operatorname{BS}^{\prime}:= \pi\circ \operatorname{BS}:L^2\mathcal{V}_0(\Omega)\rightarrow L^2\mathcal{V}_0(\Omega) 
\end{gather}
defines a self-adjoint, compact, linear operator \cite[Theorem 9.1]{CDG00} which admits a discrete spectrum which accumulates only at $0$ and therefore admits a largest positive and smallest negative eigenvalue. The corresponding eigenfields are all smooth vector fields. As a consequence of this we find that
\begin{gather}
	\label{S1E10}
	\sup_{B\in L^2\mathcal{V}_0(\Omega)\setminus\{0\}}\frac{\mathcal{H}(B)}{\|B\|^2_{L^2(\Omega)}}=\lambda_+(\Omega)
\end{gather}
where $\lambda_+(\Omega)>0$ denotes the largest positive eigenvalue of the modified Biot-Savart operator $\operatorname{BS}^{\prime}$.
We recall that helicity $\mathcal{H}$ provides a measure of entanglement of distinct field lines and so high helicity values, and in turn high $\lambda_+(\Omega)$ values, may be regarded as a measure of topological stability of the underlying eigenfields. It further follows from (\ref{S1E2}) that the magnetic field is "frozen into the fluid" in the sense that every fluid particle which starts at a specific field line of $B$ at time $t=0$ remains on that field line for all times, a phenomenon known as Alfv\'{e}n's theorem, \cite{A42}. Therefore, the quantity $\lambda_+(\Omega)$ may be regarded as a simplified, idealised measure for the highest topological stability that any plasma configuration which is supported on $\Omega$ may achieve.

Hence, it is of interest to look for a domain $\Omega$ which maximises the quantity $\lambda_+(\Omega)$. It turns out that this is an ill-posed problem since by scaling the domain $\Omega$ one can let $\lambda_+(\Omega)$ become arbitrarily large. From a more practical point of view it is sensible to restrict attention to domains which satisfies a certain "finiteness" constraint.

The main observation in this regard are inequalities (\ref{S1E6}),(\ref{S1E7}) which tell us that $\lambda_+(\Omega)$ is bounded away from infinity as long as we fix its volume. We are therefore interested in the following iso-volumetric problem, for fixed $V>0$,
\begin{gather}
	\label{S1E11}
	\text{find }\sup_{\operatorname{vol}(\Omega)=V}\lambda_+(\Omega)
\end{gather}
where the supremum is taken among bounded smooth domains of fixed volume. This problem has been initiated in \cite{CDGT002} with the following findings
\begin{thm}[{\cite[Theorem D]{CDGT002}}]
	\label{S1T1}
	Let $\Omega\subset\mathbb{R}^3$ be a bounded smooth domain of volume $V>0$. If $\Omega$ attains the supremum in (\ref{S1E11}) among all other bounded smooth domains $\widetilde{\Omega}$ of the same volume, then the following holds
	\begin{enumerate}
		\item Every helicity maximising eigenfield $B$ of the modified Biot-Savart operator $\operatorname{BS}^{\prime}$ on $\Omega$ is of constant, non-zero, speed on the boundary, i.e. $|B||_{\partial\Omega}=\text{const}>0$.
		\item All connected components of $\partial\Omega$ are tori.
		\item The field lines of every helicity maximising eigenfield $B$ of the modified Biot-Savart operator on $\Omega$ starting at a boundary point are geodesics on $\partial\Omega$.
	\end{enumerate}
\end{thm}
The existence of optimal domains for (\ref{S1E11}) remains an open problem, but some partial results have been obtained. In \cite[Theorem 1.1]{EGPS23} it was shown that among all bounded, convex domains of fixed volume there exists a bounded convex domain which realises the supremum in (\ref{S1E11}). In addition, it was shown in \cite[Corollary 2.7]{G24OptimalHelicityDomain} that among domains satisfying a uniform ball condition and a volume constraint, there exists a domain which realises the supremum in (\ref{S1E11}).

We observe that \Cref{S1T1} provides topological conditions on optimal domains, but no geometric conditions are derived. In particular, any solid torus $S^1\times D^2\cong \Omega\subset\mathbb{R}^3$ of prescribed volume can, a priori, be an optimal domain.

The goal of the present work is to derive conditions concerning the geometry of optimal domain. More precisely, we show here (modulo some technical assumption) that if $\Omega$ is a solid torus of prescribed volume which achieves the supremum in (\ref{S1E11}) among all other bounded smooth domains of the same volume, then $\Omega$ must satisfy certain geometric conditions, see \Cref{S2T1}-\Cref{S2T3}.
\newline
\newline
Let us shortly comment on the approach that we take here. The vector fields $\operatorname{BS}(B)$ and $\operatorname{BS}^{\prime}(B)$ differ only by a gradient field \cite[Section 7]{C99} so that if $B$ is an eigenfield of $\operatorname{BS}^{\prime}$ corresponding to some $\lambda\neq 0$ we get
\begin{gather}
	\nonumber
	\lambda B=\operatorname{BS}^{\prime}(B)\Rightarrow \lambda \operatorname{curl}(B)=\operatorname{curl}\left(\operatorname{BS}^{\prime}(B)\right)=\operatorname{curl}(\operatorname{BS}(B))=B
\end{gather}
where we used in the last step that $B$ is div-free and tangent to $\partial\Omega$, see \cite[Theorem A]{CDG01}. This implies that $B$ is in particular div-free and a curl-eigenfield.

We hence deal with the question whether a curl eigenvalue $\mu$ which corresponds to a vector field $B$ on a domain $\Omega$ with $\operatorname{div}(B)=0$, $\operatorname{curl}(B)=\mu B$ for some $\mu\in \mathbb{R}$, $B\parallel \partial\Omega$ and $|B||_{\partial\Omega}=$const.$\neq 0$ can be controlled in terms of the geometry of the underlying domain $\Omega$. If such a result is available one can show by comparison with the eigenvalue of an Euclidean ball of the same volume (whose eigenvalues are explicitly known \cite[Theorem A]{CDGT00}) that optimal domains must satisfy certain geometric constraints.

The case $\mu=0$ is special and we discuss it in more detail \Cref{AppS1}. In this situation $B$ is a harmonic Neumann field, a finite dimensional vector space whose dimension coincides with the dimension of the first de Rham cohomology group of $\Omega$, \cite[Definition 2.2.1, Theorem 2.6.1]{S95}. It turns out that if $B$ is a harmonic Neumann field on a bounded, smooth Euclidean domain $\Omega\subset\mathbb{R}^3$ and of constant speed on the boundary, then it must be identically zero on all of $\Omega$. The argument is of Bochner-type (and thus extends to ambient spaces other than $\mathbb{R}^3$ as long as they satisfy certain geometric conditions) and turns out to fail as soon as $\mu\neq 0$.

In the case $\mu\neq 0$ we are not able to exclude the existence of such fields $B$ altogether, but we derive some bounds on $\mu$ in terms of the geometry of the underlying domain $\Omega$, see \Cref{S3C6}. The argument will be entirely different from the Bochner-type argument in the case $\mu=0$. For this new argument to work we require an additional condition on the field $B$. This condition will be automatically satisfied whenever $B$ is an eigenfield of the modified Biot-Savart operator, so that imposing this additional property will not be a restriction when it comes to dealing with our shape optimisation problem, cf. \Cref{S3L3} for the precise conditions which we impose on our eigenfield $B$.
\newline
\newline
Before we come to the main results of this manuscript, let us point out that recently a related shape optimisation problem has been studied in the literature which includes an additional constraint on the vector fields $B$ under consideration. More precisely, if we define
\begin{gather}
	\label{S1E12}
	\mathcal{K}(\Omega):=\{B\in L^2\mathcal{V}_0(\Omega)\mid \langle B,X\rangle_{L^2(\Omega)}=0\text{ for all }X\in L^2(\Omega,\mathbb{R}^3)\text{ with }\operatorname{curl}(X)=0\},
\end{gather}
then the curl operator becomes a self-adjoint operator with compact inverse on the domain $D_{\Omega}=\{B\in \mathcal{K}(\Omega)\mid \operatorname{curl}(B)\in\mathcal{K}(\Omega)\}$, cf. \cite{YG90}. In particular it admits a smallest positive eigenvalue $\mu_+(\Omega)>0$ and largest negative eigenvalue $\mu_-(\Omega)<0$. We have already seen that if $B\in L^2\mathcal{V}_0(\Omega)$ is an eigenfield of $\operatorname{BS}^{\prime}$ with eigenvalue $\lambda\neq 0$, then it satisfies $\operatorname{curl}(B)=\frac{B}{\lambda}$ and thus is a curl eigenfield of eigenvalue $\frac{1}{\lambda}$. We notice however that a priori eigenfields of $\operatorname{BS}^{\prime}$ need not be elements of $\mathcal{K}(\Omega)$, (\ref{S1E12}), and therefore the eigenvalues of curl with domain $D_{\Omega}$ do not coincide with the multiplicative inverses of the eigenvalues of $\operatorname{BS}^{\prime}$. We however have the following characterisation of $\mu_+(\Omega)$, \cite[Theorem 2.1]{G20},
\begin{gather}
	\label{S1E13}
	\frac{1}{\mu_+(\Omega)}=\sup_{B\in \mathcal{K}(\Omega)\setminus\{0\}}\frac{\mathcal{H}(B)}{\|B\|^2_{L^2(\Omega)}} 
\end{gather}
from which we deduce, (\ref{S1E10}), $\lambda_+(\Omega)\geq \frac{1}{\mu_+(\Omega)}$. This, in combination with (\ref{S1E7}), shows further that $\mu_+(\Omega)$ is bounded below in terms of the volume of $\Omega$ alone. It is thus natural to consider the corresponding spectral shape optimisation problem of minimising $\mu_+(\Omega)$ among bounded smooth domains of prescribed volume. This spectral iso-volumetric problem has been independently initiated in \cite[Section 2]{G20Diss} and \cite{EPS23}. Interestingly, the analogue of \Cref{S1T1} remains valid in this setting, i.e. if $\Omega$ is a bounded smooth domain which minimises $\mu_+(\Omega)$ among all other bounded smooth domains of the same volume as $\Omega$, then every curl eigenfield $B\in D_{\Omega}$ which corresponds to the eigenvalue $\mu_+(\Omega)$ is of non-zero constant speed on the boundary, its boundary orbits are geodesics on the boundary and the boundary components of $\partial\Omega$ are tori, cf. \cite[Proposition 2.2.4]{G20Diss}, \cite[Proposition 2.1, Corollary 2.3]{EPS23}.

Beyond these topological restrictions in the spirit of \Cref{S1T1} some additional geometric constrains on optimal domains for the iso-volumetric spectral problem were derived in \cite[Theorem 1.2, Corollary 1.3]{EPS23} and generalised in \cite[Theorem 2.4, Theorem 2.6]{G23}. It is in particular shown, \cite[Theorem 1.2]{EPS23}, that if $\Omega$ is rotationally symmetric with a connected boundary, then for "generic" cross-sections, generating $\Omega$, the domain $\Omega$ cannot minimise $\mu_+$ among all other bounded smooth domains of the same volume. The argument hinges crucially on two facts:
\begin{enumerate}
	\item Every first eigenfield $B$ is $L^2(\Omega)$-orthogonal to the curl-free fields since $B\in D_{\Omega}\subset \mathcal{K}(\Omega)$,
	\item We have an explicit formula for a non-zero element of the space of harmonic Neumann fields $\mathcal{H}_N(\Omega):=\{\Gamma\in L^2\mathcal{V}_0(\Omega)\mid \operatorname{curl}(\Gamma)=0\}$ whose Euclidean norm can be directly linked to geometric properties of the the domain $\Omega$.
\end{enumerate}
The generalisations obtained in \cite[Theorem 2.6]{G23} exploit the same underlying principles. These arguments break down entirely, if we instead consider domains maximising $\lambda_+$ since in this case the corresponding eigenfields of $\operatorname{BS}^{\prime}$ are not $L^2(\Omega)$-orthogonal to the curl-free fields and as such the arguments used in \cite{EPS23},\cite{G23} cannot be used to derive any geometric restrictions for optimal domains pertaining $\lambda_+$.

Therefore, all the results we obtain here regarding the geometric restrictions on optimal domains of $\lambda_+$ are novel, even in the axi-symmetric situation. It is further interesting to observe, that the arguments which we provide here to establish geometric restrictions for optimal domains of $\lambda_+$ rely not only on the fact that the eigenfields $B$ of $\operatorname{BS}^{\prime}$ corresponding to $\lambda_+$ are of constant speed on the boundary and that $\operatorname{curl}(B)$ is a multiple of $B$, but also on the fact that such eigenfields have zero toroidal circulation along the boundary of $\Omega$. In contrast, eigenfields of $\operatorname{curl}$ within $D_{\Omega}$ corresponding to $\mu_+(\Omega)$ need not satisfy this additional condition so that all the geometric restrictions which are obtained in the present work for optimal domains for $\lambda_+$ need not hold for optimal domains for the first eigenvalue $\mu_+$.

One may of course wonder whether the fact that one has to make a distinction between these two optimisation problems to derive geometric restrictions in either case are simply an artefact of the methods that are employed or whether they are of fundamental nature. More specifically one may wonder if these two optimisation problems are truly distinct in the following sense:

We have seen already that $\lambda_+(\Omega)\geq \frac{1}{\mu_+(\Omega)}$ and so if we let $\operatorname{Sub}_c(\mathbb{R}^3)$ denote the collection of all bounded, smooth domains in $\mathbb{R}^3$ and define 
\begin{gather}
	\nonumber
	\Lambda(V):=\underset{\substack{\Omega\in \operatorname{Sub}_c(\mathbb{R}^3)\\ \operatorname{vol}(\Omega)=V}}{\sup}\lambda_+(\Omega)\text{ and }\mathcal{M}(V):=\underset{\substack{\Omega\in \operatorname{Sub}_c(\mathbb{R}^3)\\ \operatorname{vol}(\Omega)=V}}{\inf}\mu_+(\Omega)
\end{gather}
we find $\Lambda(V)\geq \mathcal{M}^{-1}(V)$. The question then becomes
\begin{quest}
	\label{S1Q2}
	Is it true that for all $V>0$ we have $\Lambda(V)>\mathcal{M}^{-1}(V)$?
	\end{quest}
	We notice that due to scaling properties of the functional involved, \cite[Lemma 4.3]{EGPS23}, \Cref{S1Q2} has a positive answer to some $V>0$ if and only if the answer is positive for all $V>0$.
	
	We shall see that if $\mu_+$ admits a global minimiser, then the answer to \Cref{S1Q2} must be positive so that either these two optimisation problems are truly distinct, or otherwise $\mu_+$ does not admit any smooth minimisers.
	
	Our main insight regarding \Cref{S1Q2} can hence be stated as follows
	\begin{thm}
		\label{S1T3}
		If $\Lambda(V)=\mathcal{M}^{-1}(V)$ for some (and hence every) $V>0$, then there does not exist any bounded, smooth domain $\Omega\subset\mathbb{R}^3$ of volume $V$ with $\mu_+(\Omega)=\mathcal{M}(V)$.
	\end{thm}
	\textbf{Structure of the paper:} In \Cref{S21} we fix some notation which will be used throughout the paper. In \Cref{S22} we introduce the main results pertaining the geometric restrictions on domains maximising $\lambda_+$. In \Cref{S3} we provide proofs of the main results including a proof of \Cref{S1T3}. The appendix includes some auxiliary results.
\section{Main results}
\label{S2}
\subsection{Notation}
\label{S21}
Let throughout $\Omega\subset\mathbb{R}^3$ be a bounded smooth domain. 
\begin{itemize}
	\item We denote by $L^2\mathcal{V}_0(\Omega):=\{B\in L^2(\Omega,\mathbb{R}^3)\mid \operatorname{div}(B)=0=\mathcal{N}\cdot B\}$ where the div-free condition and tangency condition are understood in the weak sense, i.e. $B\in L^2\mathcal{V}_0(\Omega)\Leftrightarrow \int_{\Omega}B\cdot \nabla fd^3x=0$ for all $f\in H^1(\Omega)$.
	\item We let $\mathcal{K}(\Omega):=\{B\in L^2\mathcal{V}_0(\Omega)\mid \int_{\Omega}B\cdot Yd^3x=0\text{ for all }Y\in L^2(\Omega,\mathbb{R}^3)\text{ with }\operatorname{curl}(Y)=0\}$ where the curl-free condition is understood in the weak sense, i.e. $\int_{\Omega}Y\cdot \operatorname{curl}(\psi)d^3x=0$ for all $\psi\in C^{\infty}_c(\Omega,\mathbb{R}^3)$.
	\item We denote by $\mathcal{H}_N(\Omega)$ the space of harmonic Neumann fields, i.e. $\mathcal{H}_N(\Omega)=\{\Gamma\in L^2\mathcal{V}_0(\Omega)\mid \operatorname{curl}(\Gamma)=0\}$. It is well-known that $\mathcal{H}_N(\Omega)$ consists of elements which are smooth up to the boundary and that $\dim\left(\mathcal{H}_N(\Omega)\right)=\dim\left(H^1_{\operatorname{dR}}(\Omega)\right)$ where we denote by $H^k_{\operatorname{dR}}(\Omega)$ the $k$-th de Rham cohomology group of $\Omega$, cf. \cite[Theorem 2.2.7 \& Theorem 2.6.1]{S95}. Similarly, we denote by $\mathcal{H}(\partial\Omega)$ the space of harmonic fields on $\partial\Omega$, i.e. the space of smooth tangent vector fields on $\partial\Omega$ which are div- and curl-free on $\partial\Omega$.
	\item We denote by $\lambda_+(\Omega)>0$ the largest positive eigenvalue of the modified Biot-Savart operator as defined in (\ref{S1E9}). All eigenfields of $\operatorname{BS}^{\prime}$ are smooth, since they satisfy an elliptic equation, see (\ref{S3E4}), cf. \cite[Corollary 2.15]{ABDG98}.
	\item We denote by $\mu_+(\Omega)>0$, the smallest positive curl eigenvalue with domain $D_{\Omega}:=\{B\in \mathcal{K}(\Omega)\mid \operatorname{curl}(B)\in \mathcal{K}(\Omega)\}$. Similarly, all such eigenfields are smooth up to the boundary.
	\item We set $\operatorname{Sub}_c(\mathbb{R}^3):=\{\Omega\subset\mathbb{R}^3\mid\Omega\text{ is a bounded, smooth domain}\}$ and for given $V>0$ we let $\Lambda(V):=\underset{\substack{\Omega\in \operatorname{Sub}_c(\mathbb{R}^3)\\ \operatorname{vol}(\Omega)=V}}{\sup}\lambda_+(\Omega)$ and $\mathcal{M}(V):=\underset{\substack{\Omega\in \operatorname{Sub}_c(\mathbb{R}^3)\\ \operatorname{vol}(\Omega)=V}}{\inf}\mu_+(\Omega)$.
		\item We denote the surface gradient of a function $f\in C^{\infty}(\partial\Omega)$ either by $\nabla^{\partial\Omega}f$ or simply by $df$, where $\nabla^{\partial\Omega}f=\nabla \tilde{f}-(\mathcal{N}\cdot \nabla \tilde{f})\mathcal{N}$ where $\tilde{f}$ is any smooth extension of $f$ to $\overline{\Omega}$, $\mathcal{N}$ is the outward unit normal and $\nabla \tilde{f}$ is the standard Euclidean gradient.
		\item We say that $\Omega$ is a smooth solid torus if $\overline{\Omega}$ is a smooth embedding of $S^1\times D$ into $\mathbb{R}^3$ where $D$ is the (closed) unit disc in $\mathbb{R}^2$ and $S^1$ the unit circle.
		\item $|\Omega|$ and $|\partial\Omega|$ denote the volume of $\Omega$ and area of $\partial\Omega$ respectively.
		\item Let $\chi:S^1\rightarrow\mathbb{R}^3$ be a smooth embedding which we identify (with an abuse of notation) with a curve $\chi:[0,L]\rightarrow \mathbb{R}^3$ which we always assume to be arc-length parametrised (and thus $L=\mathcal{L}(\chi)$ is the length of $\chi$). We say that $\chi$ is regular, if $\chi^{\prime\prime}(s)\neq 0$ for all $0\leq s\leq L$.
		\item Let $\chi$ be a smooth regular curve and $\psi:S^1\times D\rightarrow \mathbb{R}^2$, $(t,x)\mapsto \psi(t,x)\equiv \psi_t(x)$ be a smooth family of diffeomorphisms which fixes the origin, i.e. $\psi_t(0)=0$ for all $t\in S^1$ and $\psi_t:D\rightarrow D_t\equiv \psi_t(D)\subset\mathbb{R}^2$ is a diffeomorphism onto its image for every $t\in S^1$. We can then write $\psi_t(x)=(\mu_t(x),\nu_t(x))$ for suitable functions $\mu_t,\nu_t$. We denote by $\{T,N,B\}$ the Frenet-Serret frame associated with $\chi$ and observe that
		\begin{gather}
			\label{S2E1}
			\Psi:S^1\times D\rightarrow \mathbb{R}^3\text{, }(s,x)\mapsto \chi(s)+\mu_s(x)N(s)+\nu_s(x)B(s)
		\end{gather}
		defines a diffeomorphism onto its image whenever $\sup_{t\in S^1}\operatorname{diam}(D_t)\ll 1$. More precisely this is the case as soon as $\sup_{t\in S^1}\operatorname{diam}(D_t)<\operatorname{reach}(\chi)$, where $\operatorname{reach}(\chi)$ can be defined as the largest number $\rho>0$ such that for all $x,y\in \chi([0,L])$, the disc of radius $\rho$ centred at $x$ and contained in the normal plane of $\chi$ at $x$ does not intersect the corresponding disc of radius $\rho$ at $y$. We call a smooth solid torus $\Omega\subset\mathbb{R}^3$ a regular solid torus with core orbit $\chi$ if $\Omega$ is the image of a map $\Psi$ as in (\ref{S2E1}) induced by $\chi$.
		\item We call a regular, smooth solid torus $\Omega\subset\mathbb{R}^3$ a tubular neighbourhood of radius $R$ around a core orbit $\chi$, if $\Psi$ in (\ref{S2E1}) can be taken to be $\mu_s(x_1,x_2)=R x_1$, $\nu_s(x_1,x_2)=R x_2$. In this context $\chi$ is usually referred to as a (smooth) knot and $\Omega$ is then a tubular neighbourhood of radius $R$ around the knot $\chi$.
\end{itemize}
\subsection{Main results}
\label{S22}
The first main result concerns rotationally symmetric solid tori. We say that $\Omega\subset\mathbb{R}^3$ is a standard rotationally solid torus of minor radius $0<r$, major radius $R>r$ and aspect ratio $a=\frac{R}{r}$ if
\begin{gather}
	\label{S2E2}
	\Omega=\left\{(x,y,z)\in \mathbb{R}^3 \mid (\sqrt{x^2+y^2}-R)^2+z^2\leq r^2\right\}.
\end{gather}
The following is our main result regarding axi-symmetric solid tori.
\begin{thm}
	\label{S2T1}
	Let $\Omega\subset\mathbb{R}^3$ be a smooth rotationally symmetric solid torus. Let further $d_->0$ denote the minimal distance of $\Omega$ to its axis of symmetry. If
	\begin{gather}
		\nonumber
		|\Omega|\leq d^3_-,
	\end{gather}
	then $\Omega$ does not maximise the largest positive Biot-Savart eigenvalue $\lambda_+$ among all other bounded smooth domains of the same volume and hence $\lambda_+(\Omega)<\Lambda(|\Omega|)$.
	
	In particular, if $\Omega$ is a standard rotationally symmetric solid torus of minor radius $0<r$, major radius $R>0$ and aspect ratio $a:=\frac{R}{r}$, then $\Omega$ does not maximise $\lambda_+$ whenever $a\geq 6$.
\end{thm}
We notice that in contrast to the result in \cite{EPS23} which was concerned with the first curl eigenvalue $\mu_+$, we are not able to exclude the optimality of all rotationally symmetric domains with a convex cross section. But in turn our methods generalise to non-symmetric solid tori.
\newline
\newline
Before we state our main result regarding general regular solid tori, we state a more specialised result regarding tubular neighbourhoods of knots.

To this end we let for a given regular curve $\chi$, $\kappa(s)$ denote its curvature, $\tau(s)$ denote its torsion and if $\Omega$ is a tubular neighbourhood around $\chi$ of radius $R$ we further make the following definitions
\begin{gather}
	\label{S2E3}
	\kappa_+:=\max_{0\leq s\leq L}|\kappa(s)|\text{, }\tau_+:=\max_{0\leq s\leq L}|\tau(s)|\text{ and }\eta:=\sqrt{\frac{(1-\kappa_+R)^2}{(1-\kappa_+R)^2+R^2\tau^2_+}}.
\end{gather}
Our second main result can then be stated as follows
\begin{thm}
	\label{S2T2}
	Let $\chi$ be a regular smooth curve and let $\Omega\subset\mathbb{R}^3$ be a tubular neighbourhood of radius $R$ around $\chi$. If
	\begin{gather}
		\label{S2E4}
		\frac{L}{R}\geq \frac{223}{\eta^3}
	\end{gather}
	then $\Omega$ does not maximise $\lambda_+$ among all other bounded smooth domains of the same volume as $\Omega$, where $L$ denotes the length of $\chi$ and $\eta$ is defined in (\ref{S2E3}). In particular $\lambda_+(\Omega)<\Lambda(|\Omega|)$.
\end{thm}
We notice that the quotient $\rho(\Omega):=\frac{L}{R}$ is known as the rope length of $\Omega$ and that therefore (\ref{S2E4}) imposes a constraint on the rope length of $\Omega$ in terms of geometric properties of the core orbit. We observe that if $\chi$ is planar, i.e. contained in an Euclidean plane, then $\tau_+=0$ and hence $\eta=1$, so that in this case the condition reads $\rho(\Omega)\geq 223$. In addition, $\eta\rightarrow 1$ as $R\searrow 0$, while $\frac{L}{R}\rightarrow \infty$ as $R\searrow 0$. Hence, if the thickness $R$ of a tubular neighbourhood $\Omega$ around a regular smooth knot $\chi$ becomes too small, condition (\ref{S2E4}) is satisfied and hence $\Omega$ cannot maximise $\lambda_+$.

Moreover, the rope-length of knots has been investigated in detail in the past, see for example \cite{GMSvdM02},\cite{CKS02}, and in particular lower bounds exist on the rope length of a knot in terms of other knot theoretic quantities, cf. \cite[Theorem A]{CDG00},\cite[Theorem 1, Theorem 2 \& Corollay 2.1]{BS99}. Notably \cite[Corollary 2.1]{BS99} implies
\begin{gather}
	\label{S2E5}
	\rho(\Omega)\geq \max\left\{4\sqrt{\pi}\sqrt{\operatorname{Cr}([\chi])},\left(\frac{4\pi}{11}\operatorname{Cr}([\chi])\right)^{\frac{3}{4}}\right\}
\end{gather}
where $\operatorname{Cr}([\chi])$ denotes the crossing number of the knot $\chi$. We hence see that \Cref{S2T2} implies that if the the crossing number of a specific knot $\chi$ is large enough with respect to the maximal torsion of $\chi$ and the width $R$ of the tubular neighbourhood $\Omega$ around $\chi$, then $\Omega$ is not an optimal domain.

We further observe that we may apply \Cref{S2T2} to the standard rotationally symmetric torus of minor radius $r$, major radius $R$ and aspect ratio $a=\frac{R}{r}$ and find the condition $a\geq \frac{223}{2\pi}=35.49\dots>6$ which gives a strictly worse condition for non-optimality in comparison to \Cref{S2T1}. This is due to the fact that in the derivation of \Cref{S2T1} we explicitly exploit the axi-symmetry to derive the bound, while the construction in \Cref{S2T2} is more general and does not exploit any continuous symmetry and as a result is more general but gives worse bounds for axi-symmetric domains.
\newline
\newline
For our final main result we recall that by definition every regular solid torus is generated by a family of diffeomorphisms $\psi(s,\cdot)\equiv\psi_s:D\rightarrow \mathbb{R}^2$ and we defined $D_s:=\psi_s(D)$. We further equip the unit disc $D$ with polar coordinates $(\rho,\phi)$ and view in the following $\psi_s$ as a function of $(s,\rho,\phi)$. We then use the following additional notation for a given regular torus $\Omega\subset\mathbb{R}^3$ and smooth vector field $B$ on $\Omega$
\begin{gather}
	\label{S2E6}
	B_+:=\max_{x\in \partial\Omega}|B(x)|\text{, }\Pi:=\sup_{s\in [0,L]}\operatorname{Per}(D_s)\text{, }\delta:=\sup_{s\in [0,L]}\operatorname{diam}(D_s)\text{, }\beta:=\underset{\substack{s\in [0,L]\\ \phi\in [0,2\pi]}}{\max}|(\partial_s\psi)(s,1,\phi)|\text{, }
	\\
	\label{S2E7}
	\alpha_-:=\underset{\substack{s\in [0,L]\\ \phi\in [0,2\pi]}}{\min}|(\partial_{\phi}\psi)(s,1,\phi)|\text{, }\alpha_+:=\underset{\substack{s\in [0,L]\\ \phi\in [0,2\pi]}}{\max}|(\partial_{\phi}\psi)(s,1,\phi)|\text{, }\xi:=\frac{\alpha_-}{\alpha_+}\sqrt{\frac{(1-\kappa_+\delta)^2}{(1-\kappa_+\delta)^2+2(\tau^2_+\delta^2+\beta^2)}}
\end{gather}
with $\kappa_+$ and $\tau_+$ as in (\ref{S2E3}) and where $\operatorname{Per}(D_s)$ denotes the perimeter of $D_s$, i.e. the length of $\partial D_s$.

Our final main result then reads as follows
\begin{thm}
	\label{S2T3}
	Let $\chi$ be a regular smooth curve and let $\Omega\subset\mathbb{R}^3$ be a regular smooth solid torus with core orbit $\chi$. Suppose that
	\begin{gather}
		\nonumber
		\frac{\Pi L\xi (1-\kappa_+\delta)}{|\Omega|^{\frac{2}{3}}}\geq 13
	\end{gather}
	where $L$ is the length of $\chi$, $\Pi,\xi,\delta$ are as in (\ref{S2E6}),(\ref{S2E7}) and $\kappa_+$ is as in (\ref{S2E3}). Then $\Omega$ does not maximise $\lambda_+$ among all other bounded smooth domains of the same volume as $\Omega$. In particular $\lambda_+(\Omega)<\Lambda(|\Omega|)$.
\end{thm}
Even though the proofs of \Cref{S2T3} and \Cref{S2T2} are very similar, we notice that \Cref{S2T3} provides a worse bound for tubular neighbourhoods in comparison to \Cref{S2T2}. This is due to the fact that in the situation where the $D_s$ are all Euclidean discs, some inequalities may be improved and certain integrals may be computed explicitly, while in the general case we will use rough estimates to bound the same integrals. 
\section{Proofs of main results}
\label{S3}
\subsection{Proof of \Cref{S1T3}}
Before we come to the proof of \Cref{S1T3} we prove a useful lemma.
\begin{lem}
	\label{S3L1}
	Let $\Omega\subset\mathbb{R}^3$ be a bounded smooth domain. Then there exists a basis $\Gamma_1,\dots,\Gamma_N$ of $\mathcal{H}_N(\Omega)$ such that $\{\gamma_1,\dots,\gamma_N,\tilde{\gamma}_1,\dots,\tilde{\gamma}_N\}$ is a basis of $\mathcal{H}(\partial\Omega)$, where $\gamma_i,\tilde{\gamma}_i$ denote the $L^2(\partial\Omega)$-orthogonal projections of $\Gamma_i|_{\partial\Omega}$ and $\operatorname{BS}(\Gamma_i)^{\parallel}:= \operatorname{BS}(\Gamma_i)|_{\partial\Omega}-\left(\mathcal{N}\cdot \operatorname{BS}(\Gamma_i)\right)\mathcal{N}$ onto $\mathcal{H}(\partial\Omega)$ respectively.
\end{lem}
\begin{proof}[Proof of \Cref{S3L1}]
	Let $N:=\dim(\mathcal{H}(\Omega))$ and let $\Gamma_1,\dots,\Gamma_N$ be an $L^2(\Omega)$-orthogonal basis of $\mathcal{H}(\Omega)$. We observe first that if $\Gamma\in \mathcal{H}_N(\Omega)$ and $\Gamma|_{\partial\Omega}=\nabla^{\partial\Omega}f$ for some function $f\in C^{\infty}(\partial\Omega)$, then $\Gamma=0$ throughout $\Omega$. Indeed we can extend $f$ to a smooth function $f\in C^{\infty}(\overline{\Omega})$ and observe that, \cite[Theorem A]{CDG01}, $\operatorname{curl}(\operatorname{BS}(\Gamma))=\Gamma$. Thus, setting for notational simplicity $A:=\operatorname{BS}(\Gamma)$ we find
	\begin{gather}
		\nonumber
		\|\Gamma\|^2_{L^2(\Omega)}=\langle \Gamma,\operatorname{curl}(A)\rangle_{L^2(\Omega)}=\int_{\partial\Omega}(A\times \Gamma)\cdot \mathcal{N}d\sigma=\int_{\partial\Omega}(A\times \nabla f)\cdot \mathcal{N}d\sigma
		\\
		\nonumber
		=\int_{\Omega}\operatorname{curl}(A)\cdot \nabla fd^3x=\int_{\Omega}\Gamma\cdot \nabla fd^3x=0
	\end{gather}
	where we integrated by parts, used that $\Gamma$ and $\nabla f$ are curl-free and that $\Gamma|_{\partial\Omega}=(\nabla f)^{\parallel}$. In addition we have the identity $\operatorname{curl}_{\partial\Omega}(\Gamma|_{\partial\Omega})=\mathcal{N}\cdot \operatorname{curl}(\Gamma)=0$ and thus any $\Gamma\in \mathcal{H}_N(\Omega)$ has the property $\Gamma|_{\partial\Omega}=\nabla^{\partial\Omega}f+\gamma$ where $\gamma$ is the projection of $\Gamma|_{\partial\Omega}$ onto $\mathcal{H}(\partial\Omega)$ and $f\in C^{\infty}(\partial\Omega)$ is a suitable function. These considerations imply that if $\Gamma_1,\dots,\Gamma_N$ is any basis of $\mathcal{H}_N(\Omega)$, then $\gamma_1,\dots,\gamma_N$ are linearly independent.
	
	Further, if we express $\Gamma_i|_{\partial\Omega}=d f_i+\gamma_i$ then we find
	\begin{gather}
		\nonumber
		\int_{\partial\Omega}\gamma_i\cdot (\mathcal{N}\times \gamma_k)d\sigma=\int_{\partial\Omega}(d f_i+\gamma_i)\cdot (\mathcal{N}\times (d f_k+\gamma_k))d\sigma
		\\
		\label{S3E1}
		=\int_{\partial\Omega}\Gamma_i\cdot (\mathcal{N}\times \Gamma_k)d\sigma=\int_{\partial\Omega}(\Gamma_k\times \Gamma_i)\cdot \mathcal{N}d\sigma=0\text{ for all }1\leq i,k\leq N,
	\end{gather}
	where we used that $\mathcal{N}\times \cdot:\mathcal{H}(\partial\Omega)\rightarrow \mathcal{H}(\partial\Omega)$ is an isomorphism, that the exact fields are $L^2(\partial\Omega)$-orthogonal to the div-free fields and that the co-exact fields are $L^2(\partial\Omega)$-orthogonal to the curl-free fields and used Gauss' theorem and the fact that the $\Gamma_i$ are all curl-free on $\Omega$ in the last step.
	
	Next we observe that $\operatorname{curl}(\operatorname{BS}(\Gamma_i))=\Gamma_i\parallel \partial\Omega$ and thus $\operatorname{BS}(\Gamma_i)^{\parallel}$ is a curl-free field on $\partial\Omega$ so that we can similarly express $\operatorname{BS}(\Gamma_i)^{\parallel}=d h_i+\tilde{\gamma}_i$, where $\tilde{\gamma}_i$ are the projections onto $\mathcal{H}(\partial\Omega)$ and $h_i\in C^{\infty}(\partial\Omega)$ are suitable functions. We then compute in a similar spirit
	\begin{gather}
		\nonumber
		\int_{\partial\Omega}\tilde{\gamma}_i\cdot (\mathcal{N}\times \gamma_k)d\sigma=\int_{\partial\Omega}(d h_i+\tilde{\gamma}_i)\cdot (\mathcal{N}\times (d f_k+\gamma_k))d\sigma
		\\
		\nonumber
		=\int_{\partial\Omega}\operatorname{BS}(\Gamma_i)^{\parallel}\cdot (\mathcal{N}\times \Gamma_k)d\sigma=\int_{\partial\Omega}\operatorname{BS}(\Gamma_i)\cdot (\mathcal{N}\times \Gamma_k)d\sigma=\int_{\partial\Omega}(\Gamma_k\times \operatorname{BS}(\Gamma_i))\cdot\mathcal{N}d\sigma
		\\
		\label{S3E2}
		=\int_{\Omega}\operatorname{div}(\Gamma_k\times\operatorname{BS}(\Gamma_i))d^3x=\int_{\Omega}\Gamma_k\cdot \Gamma_id^3x=\delta_{ik}\text{ for all }1\leq i,k\leq N,
	\end{gather}
	where we used the vector calculus identity $\operatorname{div}(X\times Y)=\operatorname{curl}(X)\cdot Y-X\cdot \operatorname{curl}(Y)$ and the fact that the $\Gamma_i$ form an $L^2(\Omega)$-orthonormal basis of $\mathcal{H}_N(\Omega)$.
	
	We claim now that $\{\gamma_1,\dots,\gamma_N,\tilde{\gamma}_1,\dots,\tilde{\gamma}_N\}$ are linearly independent. To see this let $a_i,b_i\in \mathbb{R}$, $1\leq i\leq N$ with
	\begin{gather}
		\label{S3E3}
		\sum_{i=1}^N(a_i\gamma_i+b_i\tilde{\gamma}_i)=0.
	\end{gather}
	We can multiply (\ref{S3E3}) by $\mathcal{N}\times \gamma_k$ and integrate to deduce
	\begin{gather}
		\nonumber
		0=\sum_{i=1}^N\left(a_i\int_{\partial\Omega}\gamma_i\cdot (\mathcal{N}\times \gamma_k)d\sigma+b_i\int_{\partial\Omega}\tilde{\gamma}_i\cdot(\mathcal{N}\times \gamma_k)d\sigma\right)=b_k
	\end{gather}
	where we used (\ref{S3E1}) and (\ref{S3E2}). Since $1\leq k\leq N$ was arbitrary we see that (\ref{S3E3}) becomes $\sum_{i=1}^Na_i\gamma_i=0$. But we have previously argued that $\gamma_1,\dots,\gamma_N$ are linearly independent and thus $a_k=0$ for all $1\leq k\leq N$ which shows that $\gamma_1,\dots,\gamma_N,\tilde{\gamma}_1,\dots,\tilde{\gamma}_N$ are linearly independent.
	
	We are left with observing that $\dim(\mathcal{H}(\partial\Omega))=2\dim(\mathcal{H}(\Omega))$. This follows immediately from the fact that $2\chi(\Omega)=\chi(\partial\Omega)$, cf. \cite[Proposition 18.6.2]{Die08}, the definition of the Euler characteristic of a compact manifold $\chi(M)=\sum_{k=0}^{\dim(M)}(-1)^k\dim(H^k_{\operatorname{dR}}(M))$, the fact that $\dim(H^1_{\operatorname{dR}}(M))=\dim(\mathcal{H}_N(M))$ (notice that $\mathcal{H}_N(M)=\mathcal{H}(M)$ in case $\partial M=\emptyset$), cf. \cite[Theorem 2.6.1]{S95}, and the fact that $\dim(H^2_{\operatorname{dR}}(\Omega))=\#\partial\Omega-1$, \cite[Hodge decomposition theorem]{CDG02}, where $\#\partial\Omega$ denotes the number of connected components of $\partial\Omega$ and we keep in mind that we consider the situation where $\Omega$ is a domain, i.e. connected. This completes the proof.
\end{proof}
\begin{cor}
	\label{S3C2}
	Let $\Omega\subset\mathbb{R}^3$ be a bounded, smooth domain. If $B\in L^2\mathcal{V}_0(\Omega)$ is an eigenfield of $\operatorname{BS}^{\prime}$ and satisfies $B\in \mathcal{K}(\Omega)$, then $B|_{\partial\Omega}$ is a gradient field on $\partial\Omega$. In particular, $B|_{\partial\Omega}$ must vanish somewhere on $\partial\Omega$.
\end{cor}
\begin{proof}[Proof of \Cref{S3C2}]
	Let $\Omega\subset\mathbb{R}^3$ be a bounded smooth domain and $B\in L^2\mathcal{V}_0(\Omega)$ be an eigenfield of $\operatorname{BS}^{\prime}$ corresponding to some eigenvalue $\lambda\in \mathbb{R}$. If $\lambda=0$, then $\operatorname{BS}^{\prime}(B)=0$ and consequently, \cite[Theorem A]{CDG01}, $0=\operatorname{curl}(\operatorname{BS}^{\prime}(B))=\operatorname{curl}(\operatorname{BS}(B))=B$ where we used that $B$ is div-free, tangent to the boundary and that the modified Biot-Savart operator differs from the Biot-Savart operator only by a (curl-free) gradient field. Hence $\lambda\neq 0$. We then conclude in the same spirit
	\begin{gather}
		\label{S3E4}
		\operatorname{curl}(B)=\frac{\operatorname{curl}(\operatorname{BS}^{\prime}(B))}{\lambda}=\frac{B}{\lambda}=\mu B
	\end{gather}
	where we set $\mu:=\frac{1}{\lambda}$. Now let $\Gamma\in \mathcal{H}_N(\Omega)$ be any fixed element. As argued in the proof of \Cref{S3L1} we have $\operatorname{BS}(\Gamma)^{\parallel}=d h+\tilde{\gamma}$ where $\tilde{\gamma}$ is the projection onto $\mathcal{H}(\partial\Omega)$ and where $h\in C^{\infty}(\partial\Omega)$ is some suitable function. In addition, according to (\ref{S3E4}) we find $\operatorname{curl}(B)\parallel \partial\Omega$ and hence $B|_{\partial\Omega}$ is curl-free on $\partial\Omega$. This allows us to compute
	\begin{gather}
		\nonumber
		\int_{\partial\Omega}B\cdot (\mathcal{N}\times \tilde{\gamma} )d\sigma=\int_{\partial\Omega}B\cdot (\mathcal{N}\times \operatorname{BS}(\Gamma))d\sigma=\int_{\partial\Omega}(\operatorname{BS}(\Gamma)\times B)\cdot \mathcal{N}d\sigma
		\\
		\nonumber
		=\int_{\Omega}\operatorname{curl}(\operatorname{BS}(\Gamma))\cdot Bd^3x-\int_{\Omega}\operatorname{BS}(\Gamma)\cdot \operatorname{curl}(B)d^3x=\int_{\Omega}\Gamma\cdot Bd^3x-\mu\int_{\Omega}\operatorname{BS}^{\prime}(\Gamma)\cdot Bd^3x
		\\
		\label{S3E5}
		=\int_{\Omega}\Gamma\cdot Bd^3x-\mu \int_{\Omega}\Gamma\cdot \operatorname{BS}^{\prime}(B)d^3x=\int_{\Omega}\Gamma\cdot Bd^3x-\mu \lambda\int_{\Omega}\Gamma\cdot Bd^3x=0
	\end{gather}
	where we used (\ref{S3E4}), that $B$ is an eigenfield of $\operatorname{BS}^{\prime}$, that $\operatorname{BS}^{\prime}$ is symmetric and that $\mu=\frac{1}{\lambda}$. We notice that up to this point we have not used the assumption $B\in \mathcal{K}(\Omega)$ so that (\ref{S3E5}) is valid for any eigenfield of $\operatorname{BS}^{\prime}$.
	
	Now we additionally assume that $B\in \mathcal{K}(\Omega)$ which implies that $\int_{\Omega}B\cdot \Gamma d^3x=0$ for every $\Gamma\in \mathcal{H}_N(\Omega)$. We recall that we can write $\Gamma|_{\partial\Omega}=d f+\gamma$ where $\gamma$ is the projection onto $\mathcal{H}(\partial\Omega)$ and where $f\in C^{\infty}(\partial\Omega)$ is a suitable smooth function. Keeping in mind that $B|_{\partial\Omega}$ is closed, we find
	\begin{gather}
		\label{S3E6}
		\int_{\partial\Omega}B\cdot (\mathcal{N}\times \gamma)d\sigma=\int_{\partial\Omega}B\cdot (\mathcal{N}\times \Gamma)d\sigma=-\int_{\Omega}\Gamma\cdot \operatorname{curl}(B)d^3x=-\mu\int_{\Omega}\Gamma\cdot Bd^3x=0
	\end{gather}
	where we used (\ref{S3E4}) and the fact that $B$ is $L^2(\Omega)$-orthogonal to all curl-free fields.

According to \Cref{S3L1} we can fix a basis $\Gamma_1,\dots,\Gamma_N$ of $\mathcal{H}_N(\Omega)$ such that $\gamma_1,\dots,\gamma_N,\tilde{\gamma}_1,\dots,\tilde{\gamma}_N$ is a basis of $\mathcal{H}(\partial\Omega)$ where $\gamma_i$ and $\tilde{\gamma}_i$ are the projections of $\Gamma_i|_{\partial\Omega}$ and $\operatorname{BS}(\Gamma_i)^{\parallel}$ onto $\mathcal{H}(\partial\Omega)$. Since $\mathcal{N}\times\cdot:\mathcal{H}(\partial\Omega)\rightarrow \mathcal{H}(\partial\Omega)$ is an isomorphism we deduce that $\mathcal{N}\times \gamma_1,\dots,\mathcal{N}\times \gamma_N,\mathcal{N}\times \tilde{\gamma}_1,\dots,\mathcal{N}\times \tilde{\gamma}_N$ also provides a basis. Then according to (\ref{S3E5}) and (\ref{S3E6}) we deduce that $B|_{\partial\Omega}$ is $L^2(\partial\Omega)$-orthogonal to $\mathcal{H}(\partial\Omega)$. Lastly we notice that we have argued before that $B|_{\partial\Omega}$ is curl-free and hence must be a gradient field as claimed.
\end{proof}
\begin{proof}[Proof of \Cref{S1T3}]
	We argue by contradiction. Suppose that $\Lambda(V)=\mathcal{M}^{-1}(V)$ and that there does exist some bounded smooth domain $\Omega$ of volume $V$ with $\mu_+(\Omega)=\mathcal{M}(V)$. Then due to the inequality $\lambda_+(\Omega)\geq \frac{1}{\mu_+(\Omega)}=\mathcal{M}^{-1}(V)=\Lambda(V)$ we find $\lambda_+(\Omega)=\Lambda(V)$ since $\Lambda(V)$ is the supremum of $\lambda_+$ over all bounded smooth domains of volume $V$. We conclude that $\lambda_+(\Omega)=\frac{1}{\mu_+(\Omega)}$. According to the variational characterisation of $\mu_+(\Omega)$ and $\lambda_+(\Omega)$ we find
	\begin{gather}
		\nonumber
		\sup_{B\in L^2\mathcal{V}_0(\Omega)\setminus\{0\}}\frac{\mathcal{H}(B)}{\|B\|^2_{L^2(\Omega)}}=\lambda_+(\Omega)=\frac{1}{\mu_+(\Omega)}=\sup_{B\in \mathcal{K}(\Omega)\setminus\{0\}}\frac{\mathcal{H}(B)}{\|B\|^2_{L^2(\Omega)}}
	\end{gather}
	and we know that the supremum is attained on both sides precisely by the eigenfields of $\operatorname{BS}^{\prime}$ corresponding to $\lambda_+(\Omega)$ and the eigenfields of curl corresponding to $\mu_+(\Omega)$ respectively. We conclude that any first eigenfield $B\in \mathcal{K}(\Omega)$ realises the Rayleigh quotient for $\lambda_+(\Omega)$ and hence must be an eigenfield of $\operatorname{BS}^{\prime}$. Hence \Cref{S3C2} implies that $B|_{\partial\Omega}$ admits a zero. This contradicts \Cref{S1T1}, which states that $|B|$ must be a non-zero constant on $\partial\Omega$ because $\lambda_+(\Omega)=\Lambda(V)$. Hence no bounded, smooth $\Omega$ of volume $V$ can realise $\mathcal{M}(V)$ as claimed.
\end{proof}
\subsection{The axi-symmetric case}
\begin{proof}[Proof of \Cref{S2T1}]
	Let $\Omega$ be a smooth, rotationally symmetric solid torus (not necessarily a standard torus) which maximises $\lambda_+(\Omega)$ among all other bounded smooth domains of the same volume. Then, if we let $B\in L^2\mathcal{V}_0(\Omega)$ denote a (smooth) eigenfield of $\operatorname{BS}^{\prime}$, it follows from \Cref{S1T1} that $B|_{\partial\Omega}$ is of constant, non-zero speed and all field lines of $B|_{\partial\Omega}$ are geodesics on $\partial\Omega$.
	
	Upon applying isometries to $\Omega$ we may suppose that the axis of symmetry is the $z$-axis. Since $\Omega$ is a solid torus, it does not intersect the axis of symmetry and thus we may define $Y(x,y,z):=(-y,x,0)$ whose flow induces rotations around the $z$-axis so that $Y\parallel\partial\Omega$. In turn we may define $\Gamma:=\frac{Y}{|Y|^2}$ and notice that $\Gamma\in \mathcal{H}_N(\Omega)$. It is also easy to verify that in fact $\Gamma|_{\partial\Omega}\in \mathcal{H}(\partial\Omega)$ so that $\Gamma|_{\partial\Omega}=\gamma$ where $\gamma$ denotes the $L^2(\partial\Omega)$-orthogonal projection of $\Gamma|_{\partial\Omega}$ onto $\mathcal{H}(\partial\Omega)$. We may further define $\tilde{\gamma}:=\mathcal{N}\times \gamma\in \mathcal{H}(\partial\Omega)$ and together $\gamma,\tilde{\gamma}$ span $\mathcal{H}(\partial\Omega)$. We recall that according to (\ref{S3E4}) we have $\operatorname{curl}(B)\parallel \partial\Omega$ and thus $B|_{\partial\Omega}$ is closed so that
	\begin{gather}
		\label{S3E7}
		B|_{\partial\Omega}=dh+a\gamma+b\tilde{\gamma}
	\end{gather}
	for suitable $h\in C^{\infty}(\partial\Omega)$, $a,b\in \mathbb{R}$. Next we may let $\sigma_t$ be any fixed field line of $\gamma$ and notice that any such $\sigma_t$ bounds a surface outside of $\Omega$, i.e. there exists a smooth (bounded) surface $\mathcal{A}\subset \mathbb{R}^3\setminus \overline{\Omega}$ with $\partial\mathcal{A}=\sigma_t$. Since $\sigma_t$ is a field line of $\gamma$ and $\tilde{\gamma}$ and $\gamma$ are pointwise orthogonal we find
	\begin{gather}
		\nonumber
		\lambda_+(\Omega)a\int_{\sigma_t}\gamma=\lambda_+(\Omega)\int_{\sigma_t}(dh+a\gamma+b\tilde{\gamma})=\int_{\sigma_t}(\lambda_+(\Omega)B)
		\\
		\label{S3E8}
		=\int_{\sigma_t}\operatorname{BS}^{\prime}(B)=\int_{\sigma_t}\operatorname{BS}(B)=\int_{\mathcal{A}}\operatorname{curl}(\operatorname{BS}(B))\cdot n d\sigma=0
	\end{gather}
	where we used that $\sigma_t$ is a closed loop, the eigenfield property of $B$, the fact that $\operatorname{BS}^{\prime}(B)$ and $\operatorname{BS}(B)$ differ only by a gradient field, the Stokes theorem, the fact that $\operatorname{curl}(\operatorname{BS}(B))=0$ for every div-free field which is tangent to the boundary of $\Omega$, cf. \cite[Proposition 1]{CDG01}, and the fact that $\mathcal{A}\subset \overline{\Omega}^c$. Since $\sigma_t$ is a field line of $\gamma$ we get $\int_{\sigma_t}\gamma\neq 0$ so that (\ref{S3E8}) implies $a=0$ and hence (\ref{S3E7}) becomes
	\begin{gather}
		\label{S3E9}
		B|_{\partial\Omega}=dh+b\tilde{\gamma}.
	\end{gather}
	We observe now that on the one hand $b\neq 0$ because $B$ is non-vanishing on $\partial\Omega$ according to \Cref{S1T1} and on the other hand there must be some $p\in \partial\Omega$ with $dh(p)=0$. We may then define the vector field $Z(q):=\frac{B(q)}{b|\tilde{\gamma}(p)|}$ which is well-defined because $\gamma$ and $\tilde{\gamma}$ are non-vanishing on $\partial\Omega$. Further, one easily verifies that the field lines of $\frac{\tilde{\gamma}}{|\tilde{\gamma}|}$ are all unit speed geodesics. We notice that by choice of $p$ and according to (\ref{S3E9}) we have $Z(p)=\frac{\tilde{\gamma}(p)}{|\tilde{\gamma}(p)|}$. In addition, the field lines of $B$, cf. \Cref{S1T1}, and hence of $Z$ are geodesics and so are the field lines of $\frac{\tilde{\gamma}(p)}{|\tilde{\gamma}(p)|}$. We conclude by uniqueness of geodesics that the field line of $Z$ starting at $p$ must coincide with the field line of $\frac{\tilde{\gamma}}{|\tilde{\gamma}|}$ starting at $p$. In particular, $B|_{\partial\Omega}$ admits a closed field line $\sigma_p$ which defines the boundary of an appropriate cross section $\Sigma$ of $\Omega$. With this knowledge we may compute
	\begin{gather}
		\label{S3E10}
		Lc=\int_0^{L}|B|=\int_{\sigma_p}B=b\int_{\sigma_p}\tilde{\gamma}
	\end{gather}
	where $L$ is the length of the boundary curve of $\Sigma$ and $c=|B||_{\partial\Omega}$ is the value of the constant speed of $B$ on $\partial\Omega$. We can now integrate (\ref{S3E9}) to compute $b$
	\begin{gather}
		\nonumber
		\lambda_+(\Omega)b\int_{\sigma_p}\tilde{\gamma}=\int_{\sigma_p}(\lambda_+(\Omega)B)=\int_{\sigma_p}\operatorname{BS}^{\prime}(B)=\int_{\sigma_p}\operatorname{BS}(B)=\int_{\Sigma}B\cdot n d\sigma=\operatorname{Flux}(B)
	\end{gather}
	where we used that $\sigma_p$ is a closed loop, that $B$ is an eigenfield of $\operatorname{BS}^{\prime}$, that $\operatorname{BS}^{\prime}(B)$ and $\operatorname{BS}(B)$ differ only by a gradient field, the Stokes' theorem and the fact that $\sigma_p$ bounds a cross-sectional area $\Sigma\subset \Omega$. We can insert this into (\ref{S3E10}) and arrive at
	\begin{gather}
		\nonumber
		Lc=\frac{\operatorname{Flux}(B)}{\lambda_+(\Omega)}.
	\end{gather}
	It then follows from (\ref{S3E4}) and an identical reasoning as in \cite[Lemma 3.5]{G23} that we have the identity $c=\frac{\|B\|_{L^2(\Omega)}}{\sqrt{3|\Omega|}}$ so that we arrive at
	\begin{gather}
		\label{S3E11}
		\lambda_+(\Omega)L\|B\|_{L^2(\Omega)}=\operatorname{Flux}(B)\sqrt{3|\Omega|}.
	\end{gather}
	According to the Hodge-decomposition theorem \cite[Corollary 3.5.2]{S95} we may express $B=\operatorname{curl}(A)+\widetilde{\Gamma}$ for suitable $\widetilde{\Gamma}\in \mathcal{H}_N(\Omega)$ and a div-free, smooth vector field $A$ with $A\perp\partial\Omega$. We notice that $\operatorname{Flux}(\operatorname{curl}(A))=\int_{\Sigma}\operatorname{curl}(A)\cdot nd\sigma=\int_{\sigma_p}A=0$ because $\sigma_p$ is tangent to $\partial\Omega$, while $A$ is normal to it. Thus $\operatorname{Flux}(B)=\operatorname{Flux}(\widetilde{\Gamma})$. Further, since $\dim\left(\mathcal{H}_N(\Omega)\right)=\dim\left(H^1_{\operatorname{dR}}(\Omega)\right)=1$, cf. \cite[Theorem 2.6.1]{S95} and because $\Omega$ is a solid torus, we conclude that $\widetilde{\Gamma}$ is a constant multiple of $\Gamma=\frac{Y}{|Y|^2}$. More precisely we find $\frac{\widetilde{\Gamma}}{\|\widetilde{\Gamma}\|_{L^2(\Omega)}}=\pm \frac{\Gamma}{\|\Gamma\|_{L^2(\Omega)}}$. Since $|Y|$ is rotationally symmetric we further obtain
	\begin{gather}
		\nonumber
		\|\Gamma\|^2_{L^2(\Omega)}=2\pi \int_{\Sigma}|\Gamma|d\sigma=2\pi \int_{\Sigma}\Gamma \cdot nd\sigma=2\pi\operatorname{Flux}(\Gamma)
	\end{gather}
	where we used that $\Gamma$ is everywhere normal to the cross section $\Sigma$ and hence $n=\frac{\Gamma}{|\Gamma|}$. We conclude that
	\begin{gather}
		\nonumber
		|\operatorname{Flux}(B)|=|\operatorname{Flux}(\widetilde{\Gamma})|=\|\widetilde{\Gamma}\|_{L^2(\Omega)}\left|\operatorname{Flux}\left(\frac{\widetilde{\Gamma}}{\|\widetilde{\Gamma}\|_{L^2(\Omega)}}\right)\right|
		\\
		\nonumber
		=\|\widetilde{\Gamma}\|_{L^2(\Omega)}\left|\operatorname{Flux}\left(\frac{\Gamma}{\|\Gamma\|_{L^2(\Omega)}}\right)\right|=\frac{\|\widetilde{\Gamma}\|_{L^2(\Omega)}\|\Gamma\|_{L^2(\Omega)}}{2\pi}.
	\end{gather}
	We thus deduce from (\ref{S3E11})
	\begin{gather}
		\label{S3E12}
		\lambda_+(\Omega)L\|B\|_{L^2(\Omega)}\leq \frac{\sqrt{3|\Omega|}}{2\pi}\|\widetilde{\Gamma}\|_{L^2(\Omega)}\|\Gamma\|_{L^2(\Omega)}.
	\end{gather}
	Let us set for notational simplicity $X:=\operatorname{curl}(A)$ so that $B=X+\widetilde{\Gamma}$ and we notice that the Hodge-decomposition is $L^2(\Omega)$-orthogonal so that we find $\|\widetilde{\Gamma}\|_{L^2(\Omega)}=\alpha \|B\|_{L^2(\Omega)}$ for some $0\leq \alpha\leq 1$. To estimate $\alpha$ we make use of the eigenfield equation $\lambda_+B=\operatorname{BS}^{\prime}(B)$ which implies by linearity
	\begin{gather}
		\nonumber
		\lambda_+\|\widetilde{\Gamma}\|^2_{L^2(\Omega)}=\lambda_+\langle B,\widetilde{\Gamma}\rangle_{L^2}=\langle \operatorname{BS}^{\prime}(B),\widetilde{\Gamma}\rangle_{L^2}
		\\
		\label{S3E13}
		=\langle \operatorname{BS}^{\prime}(X),\widetilde{\Gamma}\rangle_{L^2}+\langle \operatorname{BS}^{\prime}(\widetilde{\Gamma}),\widetilde{\Gamma}\rangle_{L^2}=\langle \operatorname{BS}^{\prime}(X),\widetilde{\Gamma}\rangle_{L^2}+\mathcal{H}(\widetilde{\Gamma}).
	\end{gather}
	We observe first that since $\widetilde{\Gamma}$ is a multiple of $\Gamma$ and any two distinct field lines of $\Gamma$ are unlinked, it follows from the linking interpretation of helicity, cf. \cite{Arnold2014},\cite{V03}, that $\mathcal{H}(\widetilde{\Gamma})=0$. Further we can estimate $\langle \operatorname{BS}^{\prime}(X),\widetilde{\Gamma}\rangle_{L^2(\Omega)}\leq \|\operatorname{BS}^{\prime}(X)\|_{L^2(\Omega)}\|\widetilde{\Gamma}\|_{L^2(\Omega)} \leq\lambda\|X\|_{L^2(\Omega)}\|\widetilde{\Gamma}\|_{L^2(\Omega)}$ where $\lambda$ is the largest absolute value of any eigenvalue of $\operatorname{BS}^{\prime}$. Since by assumption $\Omega$ maximises $\lambda_+$ we must have $\lambda=\lambda_+(\Omega)$ (since otherwise applying any orientation reversing diffeomorphism of $\mathbb{R}^3$ to $\Omega$ would yield a solid torus of the same volume and of higher $\lambda_+$). Combining (\ref{S3E13}) with this inequality yields $\|\widetilde{\Gamma}\|^2_{L^2(\Omega)}\leq \|X\|^2_{L^2(\Omega)}$. We recall that $\|\widetilde{\Gamma}\|^2_{L^2(\Omega)}=\alpha^2\|B\|^2_{L^2(\Omega)}$ and since the Hodge-decomposition is $L^2$-orthogonal we must have $\|X\|^2_{L^2(\Omega)}=(1-\alpha^2)\|B\|^2_{L^2(\Omega)}$ and hence $2\alpha^2\leq 1$, i.e. $\alpha\leq \frac{1}{\sqrt{2}}$. We can utilise this inequality and derive from (\ref{S3E12}) the following bound
	\begin{gather}
		\label{S3E14}
		\frac{\lambda_+(\Omega)}{|\Omega|^{\frac{1}{3}}}\leq \sqrt{\frac{3}{2}}\frac{|\Omega|^{\frac{1}{6}}}{2\pi L}\|\Gamma\|_{L^2(\Omega)}.
	\end{gather}
	Now we recall that we assumed w.l.o.g. that the axis of symmetry is the symmetry axis and that $\Gamma=\frac{Y}{|Y|^2}$ with $Y(x,y,z)=(-y,x,0)$. We notice that $|Y(x,y,z)|=\sqrt{x^2+y^2}$ is the distance of $(x,y,z)$ to the axis of symmetry and thus $|\Gamma(x,y,z)|=\frac{1}{|Y(x,y,z)|}\leq \frac{1}{d_-}$ where $d_-$ denotes the closest distance of $\Omega$ to the axis of symmetry. We further remark that $d_->0$ since otherwise $\Omega$ would be topologically a ball. Since $|\Gamma|$ is rotationally symmetric we conclude by axi-symmetry of $\Omega$, $\|\Gamma\|^2_{L^2(\Omega)}=2\pi \int_{\Sigma}|\Gamma|d\sigma\leq 2\pi\frac{|\Sigma|}{d_-}$ where $\Sigma$ is a cross section generating $\Omega$. Therefore we arrive at the estimate
	\begin{gather}
		\nonumber
		\frac{\lambda_+(\Omega)}{|\Omega|^{\frac{1}{3}}}\leq \sqrt{\frac{3}{4\pi}}\frac{|\Omega|^{\frac{1}{6}}}{d^{\frac{1}{2}}_-}\frac{\sqrt{|\Sigma|}}{L}.
	\end{gather}
	We recall that $L$ was the length of a poloidal boundary curve, i.e. the length of a curve bounding a cross section. We therefore find $L=|\partial \Sigma|$ and so the isoperimetric inequality yields $\frac{\sqrt{|\Sigma|}}{L}\leq \frac{1}{\sqrt{4\pi}}$ so that we finally arrive at the following estimate
	\begin{gather}
		\label{S3E15}
		\frac{\lambda_+(\Omega)}{|\Omega|^{\frac{1}{3}}}\leq \frac{\sqrt{3}}{4\pi}\frac{|\Omega|^{\frac{1}{6}}}{d^{\frac{1}{2}}_-}.
	\end{gather}
	We observe now that $\Omega$ maximises $\lambda_+$ among all bounded smooth domains of the same volume as $\Omega$ if and only if $\Omega$ maximises the quotient $\frac{\lambda_+(\Omega_*)}{|\Omega_*|^{\frac{1}{3}}}$ among all other bounded smooth domains $\Omega_*$ without any volume constraint, cf. \cite[Lemma 4.3]{EGPS23} (notice that the cited reference deals with the first curl eigenvalue, but the arguments provided therein apply mutatis mutandis to the Biot-Savart eigenvalue problem). The idea now is therefore to compare the right hand side of (\ref{S3E15}) to a domain $\Omega_*$ for which the eigenvalue is known. To this end we make use of the fact that $\lambda_+(B_1(0))=\frac{1}{x_0}$ where $x_0$ is the smallest positive solution to the equation $x=\tan(x)$, cf. \cite[Theorem A]{CDGT00}. Since balls do not support non-vanishing tangent fields it follows from \Cref{S1T1} that the unit ball $B_1(0)$ does not maximise $\lambda_+$ in its volume class. As pointed out this means that if
	$\frac{\sqrt{3}}{4\pi}\frac{|\Omega|^{\frac{1}{6}}}{d^{\frac{1}{2}}_-}\leq \frac{\lambda_+(B_1(0))}{|B_1(0)|^{\frac{1}{3}}}=\left(\frac{3}{4\pi}\right)^{\frac{1}{3}}\frac{1}{x_0}$, then according to (\ref{S3E15}) $\Omega$ cannot maximise $\lambda_+$ in its own volume class. Now we observe that $x_0=4.49\dots$ and so if $\frac{|\Omega|}{d^3_-}\leq 1$ the desired inequality is satisfied and hence $\Omega$ cannot be optimal.
	
	Finally, if $\Omega$ is a standard rotationally symmetric solid torus, then $|\Omega|=2\pi^2 r^2R$ and $d_-=R-r$ so that $\frac{|\Omega|}{d^3_-}=\frac{2\pi^2 a}{(a-1)^3}$ where $a=\frac{R}{r}$ is the corresponding aspect ratio. The non-optimality condition is then satisfied for all $a\geq 6$.
\end{proof}
\subsection{Regular solid tori and tubular neighbourhoods}
We first prove the following main auxiliary result
\begin{lem}
	\label{S3L3}
	Let $\Omega\subset\mathbb{R}^3$ be a regular smooth solid torus with core orbit $\chi$. Suppose that $B$ is a vector field on $\Omega$ with the following properties
	\begin{enumerate}
		\item $B$ is smooth up to and tangent to the boundary of $\Omega$, 
		\item $\operatorname{curl}(B)=\mu B$ for some $\mu\in \mathbb{R}\setminus\{0\}$,
		\item $|B|^2|_{\partial\Omega}$ is a non-zero constant,
		\item $\int_{\sigma_t}B=0$ for some toroidal loop $\sigma_t\subset\partial\Omega$.
	\end{enumerate}
	Then $\mu$ satisfies the following inequality
	\begin{gather}
		\label{S3E16}
		\frac{\Pi\xi}{\sqrt{3}|\Omega|^{\frac{1}{6}}\|X\|_{L^2(\Omega)}}\leq |\mu||\Omega|^{\frac{1}{3}} 
	\end{gather}
	for all smooth, curl-free vector fields $X$ on $\Omega$ with $\int_{\chi}X=1$ and where $\Pi$ and $\xi$ are defined as in (\ref{S2E6}) and (\ref{S2E7}). If $\Omega$ is a tubular neighbourhood of radius $R$ around $\chi$, then we may replace $\xi$ by $\eta$ in (\ref{S3E16}) where $\eta$ is defined in (\ref{S2E3}).
\end{lem}
\begin{proof}[Proof of \Cref{S3L3}]
	The goal is to attempt to follow the reasoning of the proof of \Cref{S2T1}. We notice however, that if $\Gamma\in \mathcal{H}_N(\Omega)$, then in general it will not be the case that $\Gamma|_{\partial\Omega}$ is an element of $\mathcal{H}(\partial\Omega)$ and it is also not always the case that the field lines of $\mathcal{N}\times \gamma$ are geodesics (where $\gamma$ is the projection of $\Gamma$ onto $\mathcal{H}(\partial\Omega)$) so that we no longer can exploit the fact that the field lines of our eigenfields are geodesics. Therefore several of the arguments need to be adjusted.
	
	Let for the moment $\Omega\subset\mathbb{R}^3$ be any smooth solid torus. We start by observing that there exist smooth, simple closed curves $\sigma_p,\sigma_t\subset \partial\Omega$ such that $\sigma_p$ bounds a surface $\Sigma\subset\Omega$, $\partial \Sigma=\sigma_p$, and such that $\sigma_t$ bounds a bounded surface $\mathcal{A}\subset \mathbb{R}^3\setminus\overline{\Omega}$, $\partial\mathcal{A}=\sigma_t$, and such that $\sigma_p,\sigma_t$ generate the homology of $\partial\Omega$, cf. \cite[Theorem B.2]{CP10} and \cite[Lemma 3.6]{H23} for some details. We refer to $\sigma_p$ as a poloidal loop and to $\sigma_t$ as a toroidal loop. In particular this gives us a unique basis $\gamma_p,\gamma_t\in \mathcal{H}(\partial\Omega)$ (recall that $\dim\left(\mathcal{H}(\partial\Omega)\right)=\dim\left(H^1_{\operatorname{dR}}(\partial\Omega)\right)=2$ since $\partial\Omega\cong T^2$) which is determined by
	\begin{gather}
		\label{S3E17}
		\int_{\sigma_t}\gamma_t=1=\int_{\sigma_p}\gamma_p\text{ and }\int_{\sigma_t}\gamma_p=0=\int_{\sigma_p}\gamma_t.
	\end{gather}
	We can now let $B\in L^2\mathcal{V}_0(\Omega)$ be any fixed (smooth) vector field with $\operatorname{curl}(B)=\mu B$. In particular $\operatorname{curl}(B)\parallel\partial\Omega$ and hence $B|_{\partial\Omega}$ is curl-free on $\partial\Omega$ so that we can decompose
	\begin{gather}
		\nonumber
		B|_{\partial\Omega}=dh+a\gamma_t+b\gamma_p
	\end{gather}
	for suitable $h\in C^{\infty}(\partial\Omega)$, $a,b\in \mathbb{R}$. By assumption $\int_{\sigma_t}B=0$ from which we deduce
	\begin{gather}
		\nonumber
		0=\int_{\sigma_t}B=a
	\end{gather}
	where we used that $\sigma_t$ is a closed loop and that $\int_{\sigma_t}\gamma_p=0$ so that we obtain
	\begin{gather}
		\nonumber
		B|_{\partial\Omega}=dh+b\gamma_p.
	\end{gather}
	We can integrate again along $\sigma_p$ and find
	\begin{gather}
		\nonumber
		b=\int_{\sigma_p}B=\int_{\Sigma} \operatorname{curl}(B)\cdot nd\sigma=\mu\operatorname{Flux}(B)
	\end{gather}
	where we used the eigenfield property of $B$, the fact that $\sigma_p$ bounds a surface within $\Omega$ and the normal $n$ is chosen in such a way that it is compatible with the orientation of $\sigma_p$. We therefore find
	\begin{gather}
		\label{S3E18}
		B|_{\partial\Omega}=dh+\mu\operatorname{Flux}(B)\gamma_p.
	\end{gather}
	We now let $\sigma(t)$ be any fixed field line of $B$ starting on $\partial\Omega$. Let us set $c^2:=|B|^2|_{\partial\Omega}$ which by assumption is a non-zero constant. We then denote by $\sigma[0,T]$ the field line segment starting from $t=0$ and ending at $t=T$ (which does not need to be simple for large $T$). We notice that in contrast to the situation of \Cref{S2T1} we can now not argue that $B$ must admit a poloidal field line. Instead we consider
	\begin{gather}
		\nonumber
		c^2=\frac{1}{T}\int_{\sigma[0,T]}B=\frac{h(\sigma(T))-h(\sigma(0))}{T}+\mu\operatorname{Flux}(B)\frac{\int_{\sigma[0,T]}\gamma_p}{T}
	\end{gather}
	where we take into account that the field line $\sigma$ may be non-closed. However, we observe that $h$ is bounded as a smooth function on a compact surfaces and therefore $\lim_{T\rightarrow\infty}\frac{h(\sigma(T))-h(\sigma(0))}{T}=0$. Therefore we may take the limit to deduce
	\begin{gather}
		\label{S3E19}
		|B|^2|_{\partial\Omega}=\mu\operatorname{Flux}(B)\lim_{T\rightarrow\infty}\frac{\int_{\sigma[0,T]}\gamma_p}{T}.
	\end{gather}
	Since this conclusion is valid for every sequence we may employ \Cref{AppBC2} to deduce that
	\begin{gather}
		\label{S3E20}
		|B||_{\partial\Omega}\leq \frac{\left|\mu \operatorname{Flux}(B)\right|}{\Pi\xi}
	\end{gather}
	where $\Pi=\sup_{s\in [0,L]}\operatorname{Per}(D_s)$ where $D_s=\psi_s(D)$ and $\operatorname{Per}(D_s)=\mathcal{L}(\partial D_s)$ coincides with the length of the boundary curve of the cross sections $D_s$ and where $\xi$ is given by the formula (\ref{S2E7}) and where $\xi$ may be replaced by $\eta$, as in (\ref{S2E3}), in case $\Omega$ is a tubular neighbourhood.
	
	We can now decompose $B$ according to the Hodge-decomposition theorem \cite[Corollary 3.5.2]{S95} as $B=\operatorname{curl}(A)+\Gamma$ where $A$ is a smooth, div-free vector field which is normal to $\partial\Omega$ and where $\Gamma\in \mathcal{H}_N(\Omega)$ is the $L^2(\Omega)$-orthogonal projection of $B$ onto the space of harmonic Neumann fields. In particular we then have $\operatorname{Flux}(B)=\operatorname{Flux}(\Gamma)$. Since $|B||_{\partial\Omega}$ is a non-zero constant we must have $\Gamma\neq 0$. We now claim that $\operatorname{Flux}(\Gamma)=\frac{\|\Gamma\|^2_{L^2(\Omega)}}{\int_{\sigma_t}\Gamma}$. To see this first notice that according to \cite[Proposition C.4]{G25SurfaceHelicity} we have the identity
	\begin{gather}
		\nonumber
		\operatorname{Flux}(\Gamma)=\|\Gamma\|^2_{L^2(\Omega)}\frac{\int_{\sigma_p}\widetilde{\gamma}}{\|\gamma\|^2_{L^2(\partial\Omega)}}
	\end{gather}
	where $\gamma$ denotes the $L^2(\partial\Omega)$-orthogonal projection of $\Gamma|_{\partial\Omega}$ onto $\mathcal{H}(\partial\Omega)$ and where $\widetilde{\gamma}:=\gamma\times \mathcal{N}$. We are therefore left with establishing the identity $\|\gamma\|^2_{L^2(\partial\Omega)}=\int_{\sigma_t}\gamma \cdot \int_{\sigma_p}\widetilde{\gamma}$. To this end we observe that since $\Gamma$ is not identically zero, neither is its projection $\gamma$ onto $\mathcal{H}(\partial\Omega)$, cf. the beginning of the proof of \Cref{S3L1}. It then follows from \cite[Theorem 7]{PPS22} that $\gamma$ must in fact be no-where vanishing. We can therefore define the vector fields $X:=\frac{\gamma}{|\gamma|^2}$ and $Y:=\frac{\widetilde{\gamma}}{|\gamma|^2}$. The key observation is that $X,Y$ are everywhere linearly independent, commuting vector fields so that according to the Arnold-Liouville theorem \cite[Chapter 49]{A89} there exists a diffeomorphism $\Xi:\partial\Omega\rightarrow T^2$ with the property
	\begin{gather}
		\nonumber
		\Xi_*X=a\partial_{\phi}+b\partial_{\theta}\text{ and }\Xi_*Y=c\partial_{\phi}+d\partial_{\theta}
	\end{gather}
	where $\partial_{\phi},\partial_{\theta}$ are the standard coordinate fields on $T^2$ and $a,b,c,d$ are suitable constants. By composing $\Xi$ with an additional automorphism $T^2\rightarrow T^2$ we may suppose that the field lines of $\partial_{\phi}$ are homotopic to $\Xi\circ \sigma_p$ and the field lines of $\partial_{\theta}$ are homotopic to $\Xi\circ \sigma_t$. Further, by considering $T^2=(\mathbb{R}\slash \tau_{\phi}\mathbb{Z})\times (\mathbb{R}\slash \tau_{\theta}\mathbb{Z})$ for fixed $\tau_{\phi},\tau_{\theta}>0$ it follows that, cf. the proof of \Cref{AppBL1} for more details,
	\begin{gather}
		\nonumber
		a=\tau_{\phi}\lim_{S\rightarrow\infty}\frac{\int_{\sigma_{X}[0,S]}\gamma_p}{S}\text{ and }b=\tau_{\theta}\lim_{S\rightarrow\infty}\frac{\int_{\sigma_X[0,S]}\gamma_t}{S}
	\end{gather}
	where $\sigma_X$ denotes any field line of $X$. It now follows from \cite[Section 3.4 Proof of Theorem 2.13 \& Proposition C.3]{G25SurfaceHelicity} that we have the relationships
	\begin{gather}
		\nonumber
		\gamma_t=\frac{\gamma}{\int_{\sigma_t}\gamma}\text{ and }\gamma_p=\frac{\mathcal{H}_{\partial\Omega}(\gamma_0)\gamma+\widetilde{\gamma}}{\int_{\sigma_p\widetilde{\gamma}}}
	\end{gather}
	where $\gamma_0:=\frac{\gamma}{\|\gamma\|_{L^2(\partial\Omega)}}$ and $\mathcal{H}_{\partial\Omega}$ denotes the surface helicity of a div-free field on a closed surface in $\mathbb{R}^3$, as was investigated recently in \cite{G25SurfaceHelicity} and which is a special case \cite[Theorem 2.4]{G25SurfaceHelicity} of the notion of submanifold helicity which was introduced in \cite[Section 6]{CP10}. We recall that $X=\frac{\gamma}{|\gamma|^2}$ so that a direct calculation immediately yields
	\begin{gather}
		\nonumber
		a=\tau_{\phi}\frac{\mathcal{H}_{\partial\Omega}(\gamma_0)}{\int_{\sigma_p}\widetilde{\gamma}}\text{ and }b=\frac{\tau_{\theta}}{\int_{\sigma_t}\gamma}.
	\end{gather}
	With the same reasoning, keeping in mind that $|\widetilde{\gamma}|=|\gamma|$, we find $c=\frac{\tau_{\phi}}{\int_{\sigma_p}\widetilde{\gamma}}$ and $d=0$. In particular, by making the choice $\tau_{\phi}:=\int_{\sigma_p}\widetilde{\gamma}$ and $\tau_{\theta}:=\int_{\sigma_t}\gamma$ we obtain
	\begin{gather}
		\nonumber
		\Xi_*X=\mathcal{H}_{\partial\Omega}(\gamma_0)\partial_{\phi}+\partial_{\theta}\text{ and }\Xi_*Y=\partial_{\phi}.
	\end{gather}
	We can now express the metric tensor $g$ in the coordinate system $(\phi,\theta)$ which can be achieved by observing that $X\cdot Y=0$, $|X|^2=\frac{1}{|\gamma|^2}=|Y|^2$ which gives us $3$ equations to determine the $3$ independent parameters of the metric tensor in terms of the coefficients $a,b,c,d$ of the vector field expressions $X$ and $Y$. Doing the computations yields
	\begin{gather}
		\nonumber
		g=\frac{1}{|\gamma|^2}\begin{pmatrix}
			1 & -\mathcal{H}_{\partial\Omega}(\gamma_0) \\ -\mathcal{H}_{\partial\Omega}(\gamma_0) & 1+\mathcal{H}^2_{\partial\Omega}(\gamma_0)
		\end{pmatrix}.
	\end{gather} 
	We notice that the area element $d\sigma$ on $\partial\Omega$ in this coordinate system takes the form $d\sigma=\frac{1}{|\gamma|^2}d\phi d\theta$ from which we deduce by a direct computation
	\begin{gather}
		\nonumber
		\|\gamma\|^2_{L^2(\partial\Omega)}=\int_0^{\tau_{\phi}}\int_0^{\tau_{\theta}}|\gamma|^2\frac{1}{|\gamma|^2}d\phi d\theta=\tau_{\phi}\tau_{\theta}=\int_{\sigma_p}\widetilde{\gamma}\cdot \int_{\sigma_t}\gamma=\int_{\sigma_p}\widetilde{\gamma}\cdot \int_{\sigma_t}\Gamma
	\end{gather}
	where we used that $\Gamma|_{\partial\Omega}$ and $\gamma$ differ only by a gradient field. This is precisely the claimed identity.
	
	We recall that this implies that $\operatorname{Flux}(\Gamma)=\frac{\|\Gamma\|^2_{L^2(\Omega)}}{\int_{\sigma_t}\Gamma}$ and therefore (\ref{S3E20}) becomes
	\begin{gather}
		\nonumber
		|B||_{\partial\Omega}\leq \frac{|\mu|}{\Pi\xi}\|\Gamma\|_{L^2(\Omega)}\|\Gamma_0\|_{L^2(\Omega)}
	\end{gather}
	where $\Gamma_0\in \mathcal{H}_N(\Omega)$ is the unique element satisfying $\int_{\sigma_t}\Gamma_0=1$. It then follows from \cite[Lemma 3.5]{G23} that we have $|B|^2|_{\partial\Omega}=\frac{\|B\|^2_{L^2(\Omega)}}{3|\Omega|}$. We recall that $\Gamma$ was the $L^2$-orthogonal projection of $B$ onto $\mathcal{H}_N(\Omega)$ so that we obtain
	\begin{gather}
		\label{S3E21}
		\frac{1}{\sqrt{3|\Omega|}}\leq\frac{|\mu|}{\Pi\xi}\|\Gamma_0\|_{L^2(\Omega)}.
	\end{gather}
	Finally, we observe that if $X$ is any curl-free field on $\Omega$ with the property $\int_{\sigma_t}X=1$, then according to the Hodge decomposition theorem we can write $X=\nabla h+\Gamma^{\prime}$ for suitable $h\in C^{\infty}(\overline{\Omega})$, $\Gamma^{\prime}\in \mathcal{H}_N(\Omega)$. Since $\sigma_t$ is a closed loop, we must have $\Gamma^{\prime}=\Gamma_0$ and thus $\|X\|_{L^2(\Omega)}\geq \|\Gamma_0\|_{L^2(\Omega)}$ since the Hodge-decomposition is $L^2$-orthogonal. Hence (\ref{S3E21}) becomes
	\begin{gather}
		\nonumber
		\frac{1}{\sqrt{3|\Omega|}}\leq\frac{|\mu|}{\Pi\xi}\|X\|_{L^2(\Omega)}
	\end{gather}
	which in turn can be rearranged as
	\begin{gather}
		\label{S3E22}
		\frac{\Pi\xi}{\sqrt{3}|\Omega|^{\frac{1}{6}}\|X\|_{L^2(\Omega)}}\leq |\mu||\Omega|^{\frac{1}{3}}.
	\end{gather}
	As pointed out previously, if $\Omega$ is a tubular neighbourhood, we may replace $\xi$ by $\eta$ in (\ref{S3E20}) and consequently we may also replace $\xi$ by $\eta$ in (\ref{S3E22}). Lastly, we notice that the core orbit is homotopic to a toroidal loop which completes the proof.
\end{proof}
During the proof of \Cref{S3L3} we also proved a "quantitative" version of the uniformization theorem for Euclidean tori. For potential future reference we state these conclusions as a corollary. In particular \Cref{S3C4} allows us to reconstruct an explicit lattice on $\mathbb{R}^2$ such that the corresponding metric on $\partial\Omega$ turns out to be conformally equivalent to a flat torus whose flat metric is induced by said lattice.
\begin{cor}
	\label{S3C4}
	Let $\Omega\subset\mathbb{R}^3$ be a smooth solid torus and let $\Gamma\in \mathcal{H}_N(\Omega)\setminus \{0\}$ be any fixed element. Denote by $\gamma$ the $L^2(\partial\Omega)$-orthogonal projection of $\Gamma|_{\partial\Omega}$ onto $\mathcal{H}(\partial\Omega)$ and set $\widetilde{\gamma}:=\gamma\times \mathcal{N}$. Pick further a toroidal and poloidal loop $\sigma_t$,$\sigma_p$ with $a:=\int_{\sigma_p}\widetilde{\gamma}>0$ and $b:=\int_{\sigma_t}\gamma>0$ respectively. Then there exists a diffeomorphism $\Xi:\partial\Omega\rightarrow (\mathbb{R}\slash (a\mathbb{Z}))\times (\mathbb{R}\slash (b\mathbb{Z}))$ with the property that the standard Euclidean metric on $\partial\Omega$ is the pullback via $\Xi$ of the metric $\frac{1}{|\gamma|^2}\begin{pmatrix}
		1 & -\mathcal{H}_{\partial\Omega}(\gamma_0)\\
		-\mathcal{H}_{\partial\Omega}(\gamma_0) & 1+\mathcal{H}^2_{\partial\Omega}(\gamma_0)
	\end{pmatrix}$, where $\gamma_0:=\frac{\gamma}{\|\gamma\|^2_{L^2(\partial\Omega)}}$ and where $\mathcal{H}_{\partial\Omega}(\gamma_0)$ denotes the surface helicity of $\gamma_0$.
	\newline
	Further, we have the identity $\|\gamma\|^2_{L^2(\partial\Omega)}=\left(\int_{\sigma_t}\gamma\right)\cdot \left(\int_{\sigma_p}\widetilde{\gamma}\right)$.
\end{cor}
\begin{lem}
	\label{S3L5}
	Let $\Omega\subset\mathbb{R}^3$ be a regular smooth solid torus with core orbit $\chi$. Suppose that $B$ is a vector field on $\Omega$ with the properties (i)-(iv) as stated in \Cref{S3L3}. Then there exists a smooth curl-free field $X$ on $\Omega$ such that $\int_{\chi}X=1$ and which satisfies the estimate
	\begin{gather}
		\label{a}
		\|X\|_{L^2(\Omega)}\leq \frac{\sqrt{|\Omega|}}{L(1-\kappa_+\delta)}
	\end{gather}
	where $L$ is the length of $\chi$ and where $\kappa_+$ and $\delta$ are as in (\ref{S2E3}) and (\ref{S2E6}). If in addition $\Omega$ is a tubular neighbourhood of radius $R$ around $\chi$, then we can find a curl-free $X$ with $\int_{\chi}X=1$ and $\|X\|_{L^2(\Omega)}\leq R\sqrt{\frac{2\pi}{L}}$.
\end{lem}
\begin{proof}[Proof of \Cref{S3L5}]
	We emphasise first that for a given $s\in S^1=\mathbb{R}\slash (L\mathbb{Z})$, where $L$ is the length of $\chi$, we may view the Frenet-Serre frame $\{T(s),N(s),B(s)\}$ as a basis of the tangent space at any $y\in \Psi(s,D)$ by translating the frame within the plane which is normal to $T(s)$ and contains $\chi(s)$. In this way we obtain a distinguished basis of each tangent space $T_y\Omega$ for each $y\in \Omega$ (recall that for each $y\in \Omega$ there is exactly one $s=s(y)\in [0,L)$ with $y\in \Psi(s,D)$).
	
	We then define a (smooth) $1$-form $\omega$ by demanding
	\begin{gather}
		\label{S3E24}
		\omega(y)(N(s(y)))=0=\omega(y)(B(s(y)))\text{ and }\omega(y)(T(s(y)))=\frac{1}{L(1-\kappa(s)\mu_s(x))}
	\end{gather} 
	where $x=x(y)\in D$ is the unique element with $y=\Psi(s,x)$, $\Psi$ is induced by the family of disc diffeomorphisms $\psi_s=(\mu_s,\nu_s):D\rightarrow \mathbb{R}^2$ and $L$ is the length of $\chi$. We observe first that
	\begin{gather}
		\nonumber
		\int_{\chi}\omega=L^{-1}\int_0^L(1-\kappa(s)\mu_s(x(\chi(s))))^{-1}ds=1
	\end{gather}
	where we used that by assumption the family $\psi_s$ fixes the origin so that $\Psi(s,0)=\chi(s)$ which by uniqueness gives $x(\chi(s))=0$ for all $s$ and consequently $\mu_s(x(\chi(s)))=\mu_s(0)=0$ for all $s$. To see that $\omega$ is closed, we may pick polar coordinates on $D$ which induces an "almost global" coordinate system on $\Omega$ by considering (with an abuse of notation) $\psi_s(\rho,\phi)=\psi_s(\rho\cos(\phi),\rho\sin(\phi))$ with $0<\rho<1$ and $0<\phi<2\pi$. We then obtain the coordinate system $(s,\rho,\phi)\mapsto \Psi(s,\rho,\phi)=\Psi(s,\rho\cos(\phi),\rho\sin(\phi))$ on $\Omega$ which covers $\Omega$ up to a null set. The corresponding induced basis vectors $\partial_s,\partial_{\rho},\partial_{\phi}$ are given as follows
	\begin{gather}
		\nonumber
		\partial_s=(1-\kappa \mu)T+(\tau \mu+(\partial_s\nu))B+((\partial_s\mu)-\tau \nu)N\text{, }\partial_{\rho}=(\partial_{\rho}\mu)N+(\partial_{\rho}\nu)B\text{, }\partial_{\phi}=(\partial_{\phi}\mu)N+(\partial_{\phi}\nu)B
	\end{gather}
	where $\kappa$ and $\tau$ are the curvature and torsion of $\chi$ respectively. We notice that the defining equation (\ref{S3E24}) implies that $\omega$ coincides (up to a constant factor) with the co-vector $ds$ which is the dual basis vector of $\partial_s$ with respect to the induced tangent basis. Hence $\omega$ is closed on a dense subset of $\Omega$ and thus by smoothness closed on all of $\Omega$. We can then use the musical isomorphism to identify the $1$-form $\omega$ with a smooth vector field $X$ on $\Omega$ which in turn is curl-free and satisfies $\int_{\chi}X=1$. We are hence left with estimating the $L^2$-norm of $X$. Since we have an almost global coordinate system of $\Omega$ we may compute the metric tensor in this coordinate system and a direct computation then yields
	\begin{gather}
		\label{S3E25}
		\|X\|^2_{L^2(\Omega)}=\frac{1}{L^2}\int_{\Omega}\frac{1}{(1-\kappa(s(y))\mu_{s(y)}(x(y)))^2}d^3y.
	\end{gather}
	Since $(1-\kappa \mu)\geq 1-\kappa_+\delta>0$ where $\kappa_+=\max_{s\in [0,L]}|\kappa(s)|$ and where $\delta=\sup_{s\in [0,L]}\operatorname{diam}(D_s)$ is the largest diameter among all normal cross sections of $\Omega$ we immediately obtain the estimate
	\begin{gather}
		\nonumber
		\|X\|^2_{L^2(\Omega)}\leq \frac{|\Omega|}{(1-\kappa_+\delta)^2L^2}
	\end{gather}
	which completes the proof in the case of regular smooth solid tori.
	
	In case that $\Omega$ is a tubular neighbourhood of radius $R$ we find $\psi_s(\rho,\phi)=(R\rho\cos(\phi),R\rho\sin(\phi))$ in which case (\ref{S3E25}) takes the form
	\begin{gather}
		\nonumber
		\|X\|^2_{L^2(\Omega)}=\frac{1}{L^2}\int_0^L\int_0^R\int_0^{2\pi}\frac{\rho}{(1-\kappa(s)\rho\cos(\phi))}d\phi d\rho ds.
	\end{gather}
	The above integral may be partially computed by means of a half angle substitution which yields (keep in mind that $|\kappa\rho|<1$)
	\begin{gather}
		\label{S3E26}
		\|X\|^2_{L^2(\Omega)}=\frac{2\pi R^2}{L^2}\int_0^L\frac{1-\sqrt{1-(\kappa(s)R)^2}}{(\kappa(s)R)^2}ds.
	\end{gather}
	Since we do not have an explicit formula for $\kappa(s)$ we will now estimate the remaining integral. This can be done by considering the function $f:(0,1)\rightarrow\mathbb{R}$, $z\mapsto \frac{1-\sqrt{1-z}}{z}$ where we recall that $0<|\kappa(s)R|<1$ since $\chi$ is assumed to be of non-vanishing curvature and since $\Omega$ is a tubular neighbourhood, see \cite[Lemma 2T]{BS99}. We observe now that $f$ is strictly increasing on $(0,1)$ (its derivative is strictly positive within this interval) and thus $0<f(z)<f(1)=1$ which we can use to estimate (\ref{S3E26}) as follows
	\begin{gather}
		\nonumber
		\|X\|^2_{L^2(\Omega)}\leq \frac{2\pi R^2}{L}
	\end{gather}
	which completes the proof in case $\Omega$ is a tubular neighbourhood.
	\end{proof}
	To summarise, we have proven the following main auxiliary result which allows us to bound a curl eigenvalue in terms of the geometry of the underlying domain.
	\begin{cor}
		\label{S3C6}
		Let $\Omega\subset\mathbb{R}^3$ be a regular solid torus with core orbit $\chi$. Suppose that $B$ is a vector field on $\Omega$ with properties (i)-(iv) as in \Cref{S3L3}, then the eigenvalue $\mu$ from (ii) satisfies the following inequality (with the notation from \Cref{S3L3} and \Cref{S3L5})
		\begin{gather}
			\label{S3E27}
			\frac{\Pi \xi L (1-\kappa_+\delta)}{\sqrt{3}|\Omega|^{\frac{2}{3}}}\leq |\mu||\Omega|^{\frac{1}{3}}.
		\end{gather}
		If in addition $\Omega$ is a tubular neighbourhood of radius $R$ around $\chi$, then we have the estimate
		\begin{gather}
			\label{S3E28}
			\left(\frac{L}{R}\right)^{\frac{1}{3}}\sqrt{\frac{2}{3}}\pi^{\frac{1}{3}}\eta \leq |\mu||\Omega|^{\frac{1}{3}}
		\end{gather}
		where $\eta$ is defined in (\ref{S2E3}).
	\end{cor}
	\begin{proof}[Proof of \Cref{S3C6}]
		The case of the general regular solid torus follows immediately from \Cref{S3L3} and \Cref{S3L5}. To get the estimate for the tubular neighbourhood we just need to keep in mind that $\Pi=2\pi R$ and $|\Omega|=\pi R^2L$ so that the result then also follows immediately from \Cref{S3L3} and \Cref{S3L5}.
	\end{proof}
\begin{proof}[Proof of \Cref{S2T2} and \Cref{S2T3}]
	Suppose for the moment that $\Omega\subset\mathbb{R}^3$ is a regular smooth solid torus and that it maximises $\lambda_+$ in its own volume class. According to \Cref{S1T1} there is a smooth vector field $B$ on $\Omega$ which is div-free, tangent to $\partial\Omega$, of constant (non-zero) speed on $\partial\Omega$ and such that $B$ satisfies $\operatorname{BS}^{\prime}(B)=\lambda_+(\Omega)B$ where $\operatorname{BS}^{\prime}$ denotes the modified Biot-Savart operator. It follows from (\ref{S3E4}) that $B$ also satisfies the identity $\operatorname{curl}(B)=\mu B$ with $\mu=\frac{1}{\lambda_+(\Omega)}$. Now if $\sigma_t$ is any toroidal curve on $\partial\Omega$ which bounds a surface $\mathcal{A}\subset \mathbb{R}^3\setminus\overline{\Omega}$ we can compute
	\begin{gather}
		\nonumber
		\lambda_+(\Omega)\int_{\sigma_t}B=\int_{\sigma_t}\operatorname{BS}^{\prime}(B)=\int_{\sigma_t}\operatorname{BS}(B)=\int_{\mathcal{A}}\operatorname{curl}(\operatorname{BS}(B))\cdot nd\sigma=0
	\end{gather}
	where we used that $\operatorname{BS}^{\prime}(B)$ and $\operatorname{BS}(B)$ differ only by a gradient field, that $\sigma_t$ is a closed loop, Stokes' theorem and the fact that $\operatorname{curl}(\operatorname{BS}(B))=0$ in $\overline{\Omega}^c$, cf. \cite[Proposition 1]{CDG01}. Since $\lambda_+(\Omega)>0$ we conclude $\int_{\sigma_t}B=0$ so that overall $B$ satisfies conditions (i)-(iv) of \Cref{S3L3} and thus in turn we may employ \Cref{S3C6}. It follows that
	\begin{gather}
		\label{S3E29}
		\frac{\Pi \xi L(1-\kappa_+\delta)}{\sqrt{3}|\Omega|^{\frac{2}{3}}}\leq |\mu||\Omega|^{\frac{1}{3}}=\frac{|\Omega|^{\frac{1}{3}}}{\lambda_+(\Omega)}\Leftrightarrow \frac{\lambda_+(\Omega)}{|\Omega|^{\frac{1}{3}}}\leq \frac{\sqrt{3}|\Omega|^{\frac{2}{3}}}{\Pi L\xi (1-\kappa_+\delta) }.
	\end{gather}
	Now, since $\Omega$ maximises $\lambda_+(\Omega)$, we must have $\lambda_+(\Omega)\geq \lambda_+(B_r)$ where $B_r\subset\mathbb{R}^3$ is an Euclidean ball of the same volume as $\Omega$. Since Euclidean balls do not maximise $\lambda_+$, cf. \Cref{S1T1}, the inequality is in fact strict. Further, $\lambda_+(B_r)$ is known explicitly, \cite[Theorem A]{CDGT00}, and given by $\lambda_+(B_r)=\frac{r}{x_0}$ where $x_0>0$ is the smallest positive solution of the equation $\tan(x)=x$. Since $B_r$ is of the same volume as $\Omega$, we have $r=\sqrt[3]{\frac{3}{4\pi}|\Omega|}$. We hence get from (\ref{S3E29})
	\begin{gather}
		\label{S3E30}
		\sqrt[3]{\frac{3}{4\pi}}\frac{1}{x_0}=\frac{\lambda_+(B_r)}{|\Omega|^{\frac{1}{3}}}<\frac{\lambda_+(\Omega)}{|\Omega|^{\frac{1}{3}}}\leq \frac{\sqrt{3}|\Omega|^{\frac{2}{3}}}{\Pi L\xi (1-\kappa_+\delta)}\Leftrightarrow \frac{\Pi L\xi (1-\kappa_+\delta)}{|\Omega|^{\frac{2}{3}}}<(4\pi)^{\frac{1}{3}}x_03^{\frac{1}{6}}<13
	\end{gather}
	where we used that $x_0\approx 4.49\dots$ and rounded up the approximate value to the next integer value. This shows that if the inequality in (\ref{S3E30}) is violated, the corresponding domain $\Omega$ cannot be optimal which proves \Cref{S2T3}.
	
	The proof of \Cref{S2T2} proceeds in the same way. According to \Cref{S3C6} and by comparison to the ball as above we get the inequality
	\begin{gather}
		\label{S3E31}
		\sqrt[3]{\frac{3}{4\pi}}\frac{1}{x_0}=\frac{\lambda_+(B_r)}{|\Omega|^{\frac{1}{3}}}<\frac{\lambda_+(\Omega)}{|\Omega|^{\frac{1}{3}}}\leq \left(\frac{R}{L}\right)^{\frac{1}{3}}\sqrt{\frac{3}{2}}\frac{1}{\pi^{\frac{1}{3}}\eta}\Leftrightarrow \frac{L}{R}<\left(\frac{x_0}{\eta}\right)^3\left(\frac{3}{2}\right)^{\frac{3}{2}}\frac{4}{3}<\frac{223}{\eta^3}.
	\end{gather} 
	In conclusion, if $\Omega$ is a tubular neighbourhood which violates inequality (\ref{S3E31}) then $\Omega$ cannot maximise $\lambda_+(\Omega)$ in its own volume class.
\end{proof}
\section*{Acknowledgements}
The research was supported in part by the MIUR Excellence Department Project awarded to Dipartimento di Matematica, Università di Genova, CUP D33C23001110001.
\appendix
\section{Harmonic Neumann fields}
\label{AppS1}
The goal of this section is to show that if $\Gamma\in \mathcal{H}_N(\Omega):=\{\Gamma\in H^1(\Omega,\mathbb{R}^3)\mid \operatorname{curl}(\Gamma)=0\text{, }\operatorname{div}(\Gamma)=0=\mathcal{N}\cdot \Gamma \}$ where $\Omega\subset\mathbb{R}^3$ is a bounded smooth domain then $|\Gamma|^2|_{\partial\Omega}=\text{const.}$ implies that $\Gamma=0$ in $\Omega$.
\begin{prop}
	\label{AppAP1}
	Let $\Omega\subset \mathbb{R}^3$ be a bounded smooth domain and $\Gamma\in \mathcal{H}_N(\Omega)$. If $|\Gamma|^2|_{\partial\Omega}=c$ for some $c\in [0,\infty)$, then $\Gamma=0$ throughout $\Omega$.
\end{prop}
\begin{proof}[Proof of \Cref{AppAP1}]
	We notice first that according to Bochner's formula we have
	\begin{AppA}
		\label{A1}
		\Delta \frac{|\Gamma|^2}{2}=|\nabla \Gamma|^2+\operatorname{Ric}(\Gamma,\Gamma)+\left(\nabla\operatorname{div}(\Gamma)-\operatorname{curl}(\operatorname{curl}(\Gamma))\right)\cdot \Gamma=|\nabla \Gamma|^2\geq 0
	\end{AppA}
	where $\Delta$ denotes the standard scalar Laplacian $\Delta f=\partial^2_{11}f+\partial^2_{22}f+\partial^2_{33}f$, $\operatorname{Ric}$ denotes the Ricci-tensor, which vanishes since $\mathbb{R}^3$ is flat, $|\nabla \Gamma|^2=\sum_{i,j=1}^3(\partial_i\Gamma^j)^2$ and we used that $\Gamma$ is curl- and div-free. It follows from (\ref{A1}) that $|\Gamma|^2$ is subharmonic which implies that $|\Gamma|^2$ achieves its maximum on $\partial\Omega$. By assumption $|\Gamma|^2=c$ on $\partial\Omega$ for some $c\in [0,\infty)$ (note that $c$ is assumed to have the same value on all connected components of $\partial\Omega$) from which we deduce $\max_{x\in \overline{\Omega}}|\Gamma(x)|^2=c$. It follows then identical as in the proof of \cite[Lemma 3.5]{G23} that we have the identity
	\begin{AppA}
		\label{A2}
		c=\frac{\|\Gamma\|^2_{L^2(\Omega)}}{3|\Omega|}.
	\end{AppA}
	We can combine (\ref{A2}) with the estimate $\|\Gamma\|^2_{L^{\infty}(\Omega)}=c$ to deduce
	\begin{gather}
		\nonumber
		\|\Gamma\|^2_{L^2(\Omega)}\leq c|\Omega|=\frac{\|\Gamma\|^2_{L^2(\Omega)}}{3|\Omega|}|\Omega|=\frac{\|\Gamma\|^2_{L^2(\Omega)}}{3}
	\end{gather}
	which implies $\|\Gamma\|^2_{L^2(\Omega)}\leq 0$, i.e. $\|\Gamma\|_{L^2(\Omega)}=0$ and thus $\Gamma=0$ throughout $\Omega$.
\end{proof}
We notice that if $\operatorname{curl}(B)=\mu B$ for some $\mu \neq 0$, then (\ref{A1}) becomes
\begin{gather}
	\nonumber
	\Delta \frac{|B|^2}{2}=|\nabla B|^2-\mu^2|B|^2=|\nabla B|^2-|\operatorname{curl}(B)|^2
\end{gather}
which does not need to have a non-negative sign and hence the Bochner-type arguments of \Cref{AppAP1} break down in this situation.
\section{A transversality lemma}
\label{AppS2}
We recall that if $\Omega\subset\mathbb{R}^3$ is a smooth solid torus, then there exist smooth, simple closed curves $\sigma_p,\sigma_t$ on $\partial\Omega$, such that $\sigma_p$ bounds a disc within $\Omega$ and such that $\sigma_t$ bounds a surface outside of $\Omega$ and such that there is a unique basis $\gamma_p,\gamma_t\in \mathcal{H}(\partial\Omega)$ determined by
\begin{gather}
	\nonumber
	\int_{\sigma_p}\gamma_p=1=\int_{\sigma_t}\gamma_t\text{ and }\int_{\sigma_p}\gamma_t=0=\int_{\sigma_t}\gamma_p.
\end{gather}
For a given smooth vector field $X$ defined on (and tangent to) $\partial\Omega$ we then make the following definitions
\begin{AppB}
	\label{AppBE1}
	\overline{P}(X):=\frac{\int_{\partial\Omega}\gamma_p\cdot Xd\sigma}{|\partial\Omega|}\text{, }\overline{Q}(X):=\frac{\int_{\partial\Omega}\gamma_t\cdot Xd\sigma}{|\partial\Omega|}.
\end{AppB}
\begin{lem}
	\label{AppBL1}
	Let $\Omega\subset\mathbb{R}^3$ be a smooth, solid torus and suppose that $B$ is a smooth vector field tangent to $\partial\Omega$ such that $B|_{\partial\Omega}$ is no-where vanishing and div-free on $\partial\Omega$. Then there exists a diffeomorphism $\Xi:\partial\Omega\rightarrow T^2$ and a positive smooth function $\mu\in C^{\infty}(\partial\Omega,(0,\infty))$ such that $\Xi_*(\mu B)=\overline{P}(B)\partial_{\phi}+\overline{Q}(B)\partial_{\theta}$, where $\Xi_*(\mu B)$ denotes the pushforward vector field and where $\partial_{\phi},\partial_{\theta}$ are the standard coordinate fields on $T^2$ with the property that the field lines of $\partial_{\phi}$ are homotopic to $\Xi\circ \sigma_p$ and the filed lines of $\partial_{\theta}$ are homotopic to $\Xi\circ \sigma_t$.
\end{lem}
\begin{proof}[Proof of \Cref{AppBL1}]
	We start by observing that according to \cite[Theorem 9]{PPS22} $B|_{\partial\Omega}$ is semi-rectifiable which means that there exists a diffeomorphism $\Xi:\partial\Omega\rightarrow T^2$ and a positive smooth function $\mu$ on $\partial\Omega$ such that $\Xi_*(\mu B)=a\partial_{\phi}+b\partial_{\theta}$ for suitable $a,b\in \mathbb{R}$. Since $\sigma_p,\sigma_t$ generate the first fundamental group of $\partial\Omega$ (at a common intersection point) the same is true for $\Xi\circ\sigma_p$ and $\Xi\circ \sigma_t$ so that upon composition of $\Xi$ with a suitable automorphism of $T^2$ we may w.l.o.g. assume that the field lines of $\partial_\phi$ and $\partial_\theta$ are homotopic to $\Xi\circ \sigma_p$ and $\Xi\circ \sigma_t$ respectively. We now identify $\gamma_p,\gamma_t$ with their corresponding $1$-forms via the musical isomorphisms (denoted in the same way) and notice that $(\Xi^{-1})^{\#}\gamma_p$ is a closed $1$-form where $(\Xi^{-1})^{\#}$ denotes the pullback via the inverse of $\Xi$. We can therefore express
	\begin{gather}
		\nonumber
		(\Xi^{-1})^{\#}\gamma_p=df+\alpha d\phi+\beta d\theta
	\end{gather}
	for suitable $\alpha,\beta\in \mathbb{R}$ and where $d\phi$, $d\theta$ denote the standard $1$-forms on $T^2$ which we obtain from the coordinate fields $\partial_{\phi}$, $\partial_{\theta}$ upon identification with respect to the standard flat metric on $T^2$. Now we let $\sigma_{\phi}$ denote the field line of $\partial_{\phi}$ and compute
	\begin{gather}
		\nonumber
		1=\int_{\sigma_p}\gamma_p=\int_{\Xi\circ \sigma_p}(\Xi^{-1})^{\#}\gamma_p=\int_{\sigma_{\phi}}(df+\alpha d\phi+\beta d\theta)=\alpha
	\end{gather}
	where we used that $\Xi\circ \sigma_p$ is homotopic to $\sigma_{\phi}$ that $\sigma_{\phi}$ is a closed loop, that $\int_{\sigma_{\phi}}d\theta=0$ by direct calculation and where we consider $T^2=\mathbb{R}^2\slash \mathbb{Z}^2$, i.e. assume w.l.o.g. that the variables $\phi$ and $\theta$ are $1$-periodic. Using that $\int_{\sigma_t}\gamma_p=0$ we conclude in the same fashion that $\beta=0$ and hence
	\begin{AppB}
		\label{AppBE2}
		(\Xi^{-1})^{\#}\gamma_p=df+d\phi\text{ and }(\Xi^{-1})^{\#}\gamma_t=d\widetilde{f}+d\theta
	\end{AppB}
	for suitable $f,\widetilde{f}\in C^{\infty}(T^2)$. We now let $\sigma$ be a field line of $\Xi_*B=\frac{a\partial_{\phi}+b\partial_{\theta}}{\mu\circ\Xi^{-1}}$ which are given by a reparametrisation of the curves $\widetilde{\sigma}(s)=(\phi_0+as,\theta_0+bs)$, where $(\phi_0,\theta_0)$ are the initial conditions. We notice that the reparametrisation $s(t)$ satisfies $\sigma(t)=\widetilde{\sigma}(s)$ and is a strictly increasing function with $s(t)\rightarrow \infty$ as $t\rightarrow\infty$. We can then compute according to (\ref{AppBE2})
	\begin{AppB}
		\label{AppBE3}
		\int_{\sigma[0,T]}(\Xi^{-1})^{\#}\gamma_p=\int_{\widetilde{\sigma}[0,S]}(df+d\phi)=f(\sigma(T))-f(\phi_0,\theta_0)+\int_{\widetilde{\sigma}[0,S]}d\phi=f(\sigma(T))-f(\phi_0,\theta_0)+aS.
	\end{AppB}
	On the other hand we observe that $\Xi^{-1}\circ \sigma=:\sigma_B$ is a field line of $(\Xi^{-1})_*(\Xi_*B)=B$ so that
	\begin{gather}
		\nonumber
		\int_{\sigma[0,T]}(\Xi^{-1})^{\#}\gamma_p=\int_{\sigma_B[0,T]}\gamma_p
	\end{gather}
	from which we deduce
	\begin{gather}
		\nonumber
		a=\frac{\int_{\sigma_B[0,T]}\gamma_p}{S}+\frac{f(\phi_0,\theta_0)-f(\sigma(T))}{S}.
	\end{gather}
	We observe further that $f$ is globally bounded and thus we may take the limit $T\rightarrow\infty$ (keep in mind that $S(T)\rightarrow\infty$ as $T\rightarrow\infty$) so that we obtain the formula
	\begin{AppB}
		\label{AppBE4}
		a=\lim_{T\rightarrow\infty}\frac{\int_{\sigma_B[0,T]}\gamma_p}{S(T)}.
	\end{AppB}
	The identity (\ref{AppBE4}) is valid for any initial condition $(\phi_0,\theta_0)$ and in an identical way we find
	\begin{AppB}
		\label{AppBE5}
		b=\lim_{T\rightarrow\infty}\frac{\int_{\sigma_B[0,T]}\gamma_t}{S(T)}.
	\end{AppB}
	Now since $B$ is no-where vanishing we must have either $a\neq 0$ or $b\neq 0$ (or both) and we assume w.l.o.g. that $a\neq 0$. We can then take the quotient of (\ref{AppBE4}) and (\ref{AppBE5}) to deduce
	\begin{AppB}
		\label{AppBE6}
		\frac{b}{a}=\frac{\lim_{T\rightarrow\infty}\frac{\int_{\sigma_B[0,T]}\gamma_t}{S}}{\lim_{T\rightarrow\infty}\frac{\int_{\sigma_B[0,T]}\gamma_p}{S}}=\lim_{T\rightarrow\infty}\frac{\int_{\sigma_B[0,T]}\gamma_t}{\int_{\sigma_B[0,T]}\gamma_p}=\lim_{T\rightarrow\infty}\frac{\frac{\int_{\sigma_B[0,T]}\gamma_t}{T}}{\frac{\int_{\sigma_B[0,T]}\gamma_p}{T}}.
	\end{AppB}
	Finally, we notice that $\frac{\int_{\sigma_B[0,T]}\gamma_t}{T}=\frac{\int_0^TB(\sigma_B(\tau))\cdot \gamma_t(\sigma_B(\tau))d\tau}{T}$ and that the flow of $B$ is area-preserving. It hence follows from standard ergodic theoretical results, cf. \cite[Theorem 2]{TME67}, that $\lim_{T\rightarrow\infty}\frac{\int_{\sigma_B[0,T]}\gamma_t}{T}$ exists in the $L^1(\partial\Omega)$ sense and, denoting by $\hat{f}_p$ the limit function, we further have $\frac{1}{|\partial\Omega|}\int_{\partial\Omega}\hat{f}_pd\sigma=\overline{P}(B)$, recall (\ref{AppBE1}). An analogous result is true if we replace $\gamma_p$ by $\gamma_t$ and denote its limit function by $\hat{f}_t$. It follows then from (\ref{AppBE6}) that $\frac{b}{a}\hat{f}_p=\hat{f}_t$ which we can integrate to deduce that $\frac{b}{a}\overline{P}(B)=\overline{Q}(B)$. We therefore find
	\begin{gather}
		\nonumber
		\Xi_*(\mu B)=a\partial_{\phi}+b\partial_{\theta}=a\left(\partial_{\phi}+\frac{b}{a}\partial_{\theta}\right)=\frac{a}{\overline{P}(B)}\left(\overline{P}(B)\partial_{\phi}+\overline{Q}(B)\partial_{\theta}\right).
	\end{gather}
	By adjusting the period by composition with an additional diffeomorphism we may achieve that $a=\overline{P}(B)$ which completes the proof. 
\end{proof}
Before we come to the main result of this section let us recall for convenience the following notation: Given a regular solid torus $\Omega$ with core orbit $\chi$ and which is generated by the family of disc diffeomorphisms $\psi_s$ we let $\kappa$,$\tau$ denote the curvature and torsion of $\chi$ respectively and for a given smooth vector field $B$ on $\Omega$ we set
\begin{gather}
	\nonumber
	B_+=\max_{x\in \partial\Omega}|B(x)|\text{, }\Pi=\sup_{s\in [0,L]}\operatorname{Per}(D_s)\text{, }\delta=\sup_{s\in [0,L]}\operatorname{diam}(D_s)\text{, }\beta=\underset{\substack{s\in [0,L]\\ \phi\in [0,2\pi]}}{\max}|(\partial_s\psi)(s,1,\phi)|\text{, }
	\\
	\nonumber
	\alpha_-=\underset{\substack{s\in [0,L]\\ \phi\in [0,2\pi]}}{\min}|(\partial_{\phi}\psi)(s,1,\phi)|\text{, }\alpha_+=\underset{\substack{s\in [0,L]\\ \phi\in [0,2\pi]}}{\max}|(\partial_{\phi}\psi)(s,1,\phi)|\text{, }\xi=\frac{\alpha_-}{\alpha_+}\sqrt{\frac{(1-\kappa_+\delta)^2}{(1-\kappa_+\delta)^2+2(\tau^2_+\delta^2+\beta^2)}}.
\end{gather}
Here as usual $\kappa_+=\max_{s\in [0,L]}|\kappa(s)|$ and $\tau_+:=\max_{s\in [0,L]}|\tau(s)|$.
\begin{cor}
	\label{AppBC2}
	Let $\Omega\subset\mathbb{R}^3$ be a regular smooth solid torus with core orbit $\chi$. If $B$ is a smooth vector field on $\Omega$ which is tangent to and no-where vanishing on $\partial\Omega$ and such that $\operatorname{curl}(B)\parallel \partial\Omega$ and $\int_{\sigma_t}B=0$ for a toroidal loop $\sigma_t$ on $\partial\Omega$, then there exists a strictly increasing sequence $0<T_n\rightarrow\infty$ such that for every field line $\sigma$ of $B$ on $\partial\Omega$ we have the estimate
	\begin{AppB}
		\label{AppBE7}
		\lim_{n\rightarrow\infty}\left|\frac{\int_{\sigma[0,T_n]}\gamma_p}{T_n}\right|\leq\frac{B_+}{\Pi\xi}
	\end{AppB}
	where $\gamma_p$ is the harmonic vector field induced by some fixed pair of poloidal and toroidal loops $\sigma_p,\sigma_t$ on $\partial\Omega$. If in addition $\Omega$ is a tubular neighbourhood around $\chi$, then we may replace $\xi$ in (\ref{AppBE7}) by $\eta$, as defined in (\ref{S2E3}).
\end{cor}
\begin{proof}[Proof of \Cref{AppBC2}]
	We start by letting $x(t)$ be any smooth curve in $\partial\Omega$. We then know by definition that $\Omega$ is generated by a family of smooth maps $\psi_s:D\rightarrow \mathbb{R}^2$ which are diffeomorphisms onto their image. By using standard polar coordinates $(\rho,\phi)$ on $D$ we can view $\psi_s=(\mu_s,\nu_s)$ as a function of $\rho$ and $\phi$. We notice that $\rho=1$ corresponds precisely to those points of the $D_s$ which make up $\partial\Omega$, so that we can express the curve $x(t)$ as follows
	\begin{gather}
		\nonumber
		x(t)=\chi(s(t))+\mu_{s(t)}(1,\phi(t))N(s(t))+\nu_{s(t)}(1,\phi(t))B(s(t)),
	\end{gather}
	recall (\ref{S2E1}). From now on we write with an abuse of notation $\mu_s(\phi)\equiv \mu_{s}(1,\phi)$ and similarly for $\nu$. We can then compute
	\begin{gather}
		\nonumber
		|\dot{x}(t)|^2=\dot{s}^2(1-\kappa \mu_s(\phi))^2+\left[\dot{s}\left((\partial_s\mu_s)-\nu_s\tau\right)+\dot{\phi}(\partial_{\phi}\mu_s)\right]^2+\left[\dot{s}\left((\partial_s\nu_s)+\tau\mu_s\right)+\dot{\phi}(\partial_{\phi}\nu_s)\right]^2
	\end{gather}
	where $\kappa=\kappa(s(t))$ is the curvature of $\chi$ and $\tau=\tau(s(t))$ is the torsion of $\chi$. Applying the Peter-Paul inequality yields the estimate
	\begin{AppB}
		\label{AppBE8}
		|\dot{x}(t)|^2\geq \dot{s}^2\left[\left(1-\kappa \mu_s\right)^2-\frac{1-\epsilon}{\epsilon}\left((\partial_s\mu_s)-\nu_s\tau\right)^2+\left((\partial_s\nu_s)+\tau\mu_s\right)^2\right]+\dot{\phi}^2|\partial_{\phi}\psi_s|^2(1-\epsilon)
	\end{AppB}
	for all $0<\epsilon<1$. We recall our notation $\beta=\underset{\substack{s\in [0,L]\\ \phi\in [0,2\pi]}}{\max}|(\partial_s\psi)(s,1,\phi)|$, $\delta=\sup_{s\in [0,L]}\operatorname{diam}(D_s)$ where $D_s=\psi_s(D)$, $\kappa_+=\max_{s\in [0,L]}|\kappa(s)|$ and $\tau_+=\max_{s\in [0,L]}|\tau(s)|$. We now observe that (due to the restriction we impose on the maximal diameter)
	\begin{AppB}
		\label{AppBE9}
		(1-\kappa\mu_s)\geq (1-\delta\kappa_+)>0\text{ and }(\partial_s\mu_s-\nu_s\tau)^2+(\partial_s\nu_s+\tau\mu_s)^2\leq 2\tau_+^2\delta^2+2\beta^2.
	\end{AppB}
	 With this observation it is easy to verify that for the choice $\epsilon:=\frac{2(\tau^2_+\delta^2+\beta^2)}{2(\tau^2_+\delta^2+\beta^2)+(1-\kappa_+\delta)^2}$ the square bracket next to $\dot{s}^2$ in (\ref{AppBE8}) is non-negative and thus for this specific choice of $\epsilon$ we obtain
	 \begin{AppB}
	 	\label{AppBE10}
	 	|\dot{x}(t)|^2\geq \dot{\phi}^2|\partial_\phi\psi_s|^2\frac{(1-\kappa_+\delta)^2}{(1-\kappa_+\delta)^2+2(\tau^2_+\delta^2+\beta^2)}.
	 \end{AppB}
	We notice that $(s,\phi)\mapsto |\partial_{\phi}\psi_s(1,\phi)|$ is no-where vanishing and thus $\alpha_-=\underset{\substack{s\in [0,L]\\ \phi\in [0,2\pi]}}{\min}|(\partial_\phi\psi)(s,1,\phi)|>0$. Further, we may set $\alpha_+=\underset{\substack{s\in [0,L]\\ \phi\in [0,2\pi]}}{\max}|(\partial_{\phi}\psi)(s,1,\phi)|$ and can obtain the following lower bound from (\ref{AppBE10})
	\begin{AppB}
		\label{AppBE11}
		|\dot{x}(t)|^2\geq \dot{\phi}^2\left(\frac{\alpha_-}{\alpha_+}\right)^2\frac{(1-\kappa_+\delta)^2}{(1-\kappa_+\delta)^2+2(\tau^2_+\delta^2+\beta^2)}|(\partial_{\phi}\psi_{s_*})(1,\phi(t))|^2
	\end{AppB}
	where $s_*\in [0,L]$ is any fixed value. The main conclusion now is that if we define the curve
	\begin{gather}
		\nonumber
		y(t):=\chi(s_*)+\mu_{s_*}(1,\phi(t))N(s_*)+\nu_{s_*}(1,\phi(t))B(s_*),
	\end{gather}
	then according to (\ref{AppBE11}) for every $0<T$ we have the estimate
	\begin{AppB}
		\label{AppBE12}
		\mathcal{L}(x[0,T])\geq \xi\mathcal{L}(y[0,T])\text{ where }\xi:=\frac{\alpha_-}{\alpha_+}\sqrt{\frac{(1-\kappa_+\delta)^2}{(1-\kappa_+\delta)^2+2(\tau^2_+\delta^2+\beta^2)}}
	\end{AppB}
	and where $\mathcal{L}(x[0,T])$ denotes the length of $x[0,T]$. 
	
	We now let $x(t)$ be any fixed field line of $B$. By assumption $B,\operatorname{curl}(B)\parallel \partial\Omega$ and $\int_{\sigma_t}B=0$ for a toroidal curve $\sigma_t$. This implies that $B|_{\partial\Omega}=df+a\gamma_p$ for some $f\in C^{\infty}(\partial\Omega)$ and $a\in \mathbb{R}$. This in turn implies that $X:=\mathcal{N}\times B$ is div-free on $\partial\Omega$ and $L^2(\partial\Omega)$-orthogonal to $\gamma_p$ so that $\overline{P}(X)=0$. Further, we assume that $B$, and hence $X$, are no-where vanishing. It is then a consequence of \Cref{AppBL1} that the field lines of $X$ are all closed and homotopic to $\sigma_t$ which, as a toroidal loop, is in turn (with the right choice of orientation) homotopic to a curve of the form $s\mapsto \chi(s)+\mu_s(1,\phi_*)N(s)+\nu_s(1,\phi_*)B(s)$ for some $\phi_*\in [0,2\pi)$. It then follows from the fact that $B$ is everywhere transverse to $X$ that there exists a strictly increasing sequence $T_n\rightarrow \infty$ such that if we let $x(t)$ be a field of $B$ then $\phi(T_n)=\phi(T_{n+1})$ for all $n$ and $\phi([T_n,T_{n+1}])=[0,2\pi]$, i.e. the field lines of $B$ make a full poloidal turn infinitely many times (possibly not closing on itself). This in turn implies by definition of the curve $y(t)$, that $y([0,T_n])$ made at least $n$ poloidal turns and since $y$ is a curve of fixed arc-length value $s_*$ we notice further that the length of $y$ along a full poloidal turn corresponds to the perimeter of $D_{s_*}$ so that (\ref{AppBE12}) implies $\mathcal{L}(x[0,T_n])\geq \xi n\operatorname{Per}(D_{s_*})$ for every $s_*\in [0,L]$. Hence, if we let $\Pi=\sup_{s\in [0,L]}\operatorname{Per}(D_s)$ we obtain the inequality
	\begin{AppB}
		\label{AppBE13}
		 \mathcal{L}(x[0,T_n])\geq \xi n\Pi\text{ for all }n.
	\end{AppB}
	In addition, since $x$ was chosen to be a field line of $B$, we find $\mathcal{L}(x[0,T])=\int_0^T|\dot{x}(t)|dt=\int_0^T|B(x(t))|dt\leq B_+T$ where $B_+:=\sup_{z\in \partial\Omega}|B(z)|$. Hence (\ref{AppBE13}) becomes
	\begin{AppB}
		\label{AppBE14}
		\frac{1}{T_n}\leq \frac{B_+}{n\xi \Pi}.
	\end{AppB}
	Further, by our choice of sequence, we see that if we consider the curve $x[0,T_{n}]$, then $\phi(0)=\phi(T_{n})$ and if we close the curve along the toroidal loop $z_n(s)=\chi(s)+\mu_s(1,\phi_*)N(s)+\nu_s(1,\phi_*)B(s)$ with $\phi_*=\phi(0)$ and where the range of $s$ is chosen according to the values of $s(0)$ and $s(T_{n})$ (which may, or may not, coincide), then we will obtain a (possibly non-simple) closed curve which is homotopic to $n\sigma_p$ for some poloidal curve $\sigma_p$ since by choice of the sequence $T_n$ we will have made precisely $n$ poloidal turns. In particular $\int_{x[0,T_{n}]\oplus z_n}\gamma_p=n$. The key observation is that the length of $z_n$ can be bounded independently of $n$ and therefore $\int_{x[0,T_n]}\gamma_p=n+\mathcal{O}(1)$ as $n\rightarrow\infty$. We can combine this with (\ref{AppBE14}) to deduce that
	\begin{AppB}
		\label{AppBE15}
		\limsup_{n\rightarrow\infty}\left|\frac{\int_{x[0,T_n]}\gamma_p}{T_n}\right|\leq \frac{B_+}{\Pi\xi}.
	\end{AppB}
	Lastly, we observe that if $\Omega$ is a tubular neighbourhood of radius $R$, then $\psi_s(\rho,\phi)=(R\rho\cos(\phi),R\rho\sin(\phi))$ and so $\psi_s(1,\phi)=(R\cos(\phi),R\sin(\phi))$, $\mu_s(1,\phi)=R\cos(\phi)$, $\nu_s(1,\phi)=R\sin(\phi)$. In particular $\mu^2_s(1,\phi)+\nu^2_s(1,\phi)=R^2$ and $\beta=0$, which implies that we may replace the factor $\delta$ in (\ref{AppBE9}) by $R$ and may drop the factor $2$ in the second inequality. Further, $|\partial_{\phi}\psi(s,1,\phi)|^2=R^2$ so that $\alpha_-=\alpha_+=R$ and consequently we may replace $\xi$ in (\ref{AppBE12}), and hence also in (\ref{AppBE13}),(\ref{AppBE14}) and (\ref{AppBE15}) by $\eta:=\sqrt{\frac{(1-\kappa_+R)^2}{(1-\kappa_+R)^2+\tau^2_+R^2}}$ in case that $\Omega$ is a tubular neighbourhood of radius $R$.
\end{proof}
\bibliographystyle{plain}
\bibliography{mybibfileNOHYPERLINK}
\footnotesize
\end{document}